\documentclass{amsart}
\usepackage{graphicx} 
\usepackage{amsfonts,amsmath,amssymb,amsthm}
\usepackage{xcolor}
\usepackage{hyperref}
\usepackage{rotating}
\usepackage[shortlabels]{enumitem}
\usepackage[all]{xy}
\usepackage{booktabs}
\usepackage{adjustbox}
\usepackage{appendix}

\subjclass[2020]{11G18, 14K15, 11F80}

\newcommand{\nc}[1]{\newcommand{#1}}
\nc{\on}[1]{\operatorname{#1}}
\nc{\OO}[0]{\mathcal O}
\nc{\PP}[0]{\mathbb P}

\nc{\HH}[0]{\mathcal H}

\nc{\GL}[0]{\on{GL}}
\nc{\SL}[0]{\on{SL}}
\nc{\Sp}[0]{\on{Sp}}
\nc{\GSp}[0]{\on{GSp}}
\nc{\GO}[0]{\on{GO}}
\nc{\GU}[0]{\on{GU}}
\nc{\PSL}[0]{\on{PSL}}
\nc{\PGL}[0]{\on{PGL}}
\nc{\Mat}[0]{\on{Mat}}
\nc{\GammaZeroPM}[0]{\ensuremath{\Gamma^0_{\pm}}}
\nc{\diag}[0]{\on{diag}}

\nc{\FF}[0]{\mathbb F}
\nc{\NN}[0]{\mathbb N}
\nc{\ZZ}[0]{\mathbb Z}
\nc{\QQ}[0]{\mathbb Q}
\nc{\RR}[0]{\mathbb R}
\nc{\CC}[0]{\mathbb C}
\nc{\TT}[0]{\mathbb T}

\nc{\pp}[0]{\mathfrak p}
\nc{\qq}[0]{\mathfrak q}

\nc{\ol}[1]{\overline{#1}}
\nc{\ul}[1]{\underline{#1}}

\nc{\Div}[0]{\on{Div}}
\nc{\Pic}[0]{\on{Pic}}
\nc{\Jac}[0]{\on{Jac}}

\nc{\Gal}[0]{\on{Gal}}
\nc{\Aut}[0]{\on{Aut}}
\nc{\End}[0]{\on{End}}
\nc{\EndQ}[0]{\on{End_\QQ}}
\nc{\Endo}[0]{\on{End}^0}
\nc{\EndoQ}[0]{\on{End}^0_\QQ}
\nc{\Hom}[0]{\on{Hom}}
\nc{\Frob}[0]{\on{Frob}}

\nc{\Tl}[0]{\on{T}_\ell}
\nc{\Vl}[0]{\on{V}_\ell}
\nc{\Vlambda}[0]{\ensuremath{\on{V}_\lambda}}

\nc{\Tr}[0]{\on{Tr}}

\nc{\Br}[0]{\on{Br}}

\nc{\XN}[0]{\mathcal{X}_2(N)}
\nc{\THBK}[0]{\on{THBK}}
\nc{\TB}[0]{\on{TB}}
\nc{\Grit}[0]{\on{Grit}}

\nc{\rank}[0]{\on{rank}}

\newtheorem{theorem}{Theorem}[section]
\newtheorem{proposition}[theorem]{Proposition}
\newtheorem{lemma}[theorem]{Lemma}
\newtheorem{corollary}[theorem]{Corollary}
\newtheorem{definition}[theorem]{Definition}

\newtheorem{remark}[theorem]{Remark}

\newtheorem{conjecture}[theorem]{Conjecture}

\title[Non-split Cartan curves and fake elliptic curves]{Non-split Cartan curves and a characterization of fake elliptic curves}
\author{Enric Florit}
\address{Universitat Oberta de Catalunya, Rambla del Poblenou 156, 08018, Barcelona, Spain}
\email{efloritz@uoc.edu}
\date{}
\keywords{Fake elliptic curves, modular curves, quadratic points, non-split Cartan curves}
\urladdr{https://enricflorit.com}

\begin{document}

\begin{abstract}
    The $\ell$-adic Tate module of an abelian surface with quaternionic multiplication (QM) decomposes as two copies of a two-dimensional $\QQ$-rational representation. To this day, there is no criterion to distinguish these representations from those attached to elliptic curves. For this reason, QM abelian surfaces are usually called fake elliptic curves.
    
    In this paper we characterize QM abelian surfaces defined over an imaginary quadratic field $K$. Namely, we consider an abelian surface $A/K$ without potential CM and whose $L$-function is a square. Under a reasonable conjecture, we show that $A$ has QM if and only if it has residual image contained in a non-split Cartan group modulo at least two primes. The characterization is unconditional for all indefinite quaternion discriminants up to 33. 
    
    The proof is based on previous work of Siksek and Michaud-Jacobs on quadratic points on non-split Cartan modular curves. In particular, we prove that all quadratic points on $X_{ns}(6)$, $X_{ns}(10)$ and $X_{ns}(15)$ are non-exceptional.
\end{abstract}

\maketitle

\section{Introduction}

\subsection*{Motivation}

Abelian varieties over number fields are central objects of study in number theory, and many results are centered around their parametrization via automorphic forms. These give us access to the essential properties (such as L-functions and endomorphism algebras) of the varieties without needing to explicitly write down equations.

For instance, let $N\geq 1$ be an integer. Given a classical normalized newform $f\in S_2(\Gamma_1(N))$, the Eichler-Shimura construction produces an abelian variety $A_f$ defined over $\QQ$ attached to $f$. The $L$-function of this variety is given in terms of $f$ by the formula
\[
    L(A_f,s)=\prod_{\sigma:H_f\to \CC}L({}^\sigma f,s),
\]
where $H_f := \QQ(\{a_p(f)\}_{p\nmid N})$ is a number field with the properties
\[
    \dim A_f=[H_f:\QQ] \text{ and }\End(A_{f/\QQ})\otimes\QQ =H_f.
\]
Moreover, with sufficient knowledge of the inner twists of $f$, one can compute the $\bar\QQ$-endomorphism algebra and decomposition of $A_{f,\bar\QQ}$, as done in \cite{chi87, pyle04, quer09, gq14}.

Let $K$ be an imaginary quadratic field and let $\OO_K$ be its ring of integers. Suppose for simplicity that $K$ has class number one. A Bianchi modular form is a function $\mathfrak{f}:\mathbb H_3\to \CC^3$, where $\mathbb H_3$ is the hyperbolic $3$-space, satisfying certain conditions. They correspond to automorphic forms for $\GL_2(K)$, and as in the classical modular form setting, there is a language of weights, levels, Hecke operators, and eigenforms. Given a normalized eigenform $\mathfrak{f}$, we define the field of Hecke eigenvalues $H_\mathfrak{f}=\QQ(\{c(\mathfrak p)\}_{\mathfrak p})$, where $\mathfrak p$ ranges over all primes of $K$ not dividing the level of $\mathfrak{f}$. In replacement of the Eichler-Shimura construction, we have the following conjecture.

\begin{conjecture}[Conjecture 4.1 in \cite{sengunsiksek18}]
\label{conjecture:Bianchi modularity}
    Let $\mathfrak N\subseteq \mathcal O_K$ be a nonzero ideal. Let $\mathfrak{f}$ be a weight 2 Bianchi modular form of level $\mathfrak N$ that is non-trivial and new. Suppose that $H_{\mathfrak f} = \QQ$. Then, there exists either
    \begin{enumerate}
        \item an elliptic curve $E_\mathfrak{f}/K$ of conductor $\mathfrak N$ satisfying
        \[
        \#E_{\mathfrak f}(\OO_K/\mathfrak q) = 1+\mathbf N\mathfrak q - c(\mathfrak q) \text{ for all }\mathfrak q\nmid\mathfrak N, \text{ or}
        \]
        \item a simple abelian surface $A_{\mathfrak f}/K$ such that $\End(A)\otimes\QQ$ is an indefinite division quaternion algebra with center $\QQ$. The surface $A_{\mathfrak f}$ has conductor $\mathfrak N^2$ and satisfies
        \[
        \#A_{\mathfrak f}(\OO_K/\mathfrak q) = (1+\mathbf N\mathfrak q - c(\mathfrak q))^2 \text{ for all }\mathfrak q\nmid\mathfrak N.
        \]
    \end{enumerate}
\end{conjecture}

The abelian surfaces in (2) are said to have quaternionic multiplication, or QM for short. Positive examples of both cases have been found. A particularly striking result of \cite{schembri2019examples,schembri19thesis} says that there exist abelian surfaces $A/K$ with QM satisfying (2) which are not the twist of a base change of a surface defined over $\QQ$. The conjecture is of particular interest for diophantine applications, such as the generalized Fermat equation over imaginary quadratic fields \cite{MR4101578,MR4620323,MR5057874}.

From Conjecture~\ref{conjecture:Bianchi modularity} and the evidence that supports it, \textit{the degree of the number field $H_{\mathfrak f}$ does not determine the dimension of the abelian variety attached to $\mathfrak f$}. 
The purpose of this paper is to understand what conditions determine this dimension. In other words, we give a criterion discerning between Cases (1) and (2) above.

\subsection*{Main result}

Let $K$ be a number field and consider an abelian surface $A$ defined over $K$. As above, we say $A$ has QM if $\End(A)\otimes\QQ$ is an indefinite division quaternion algebra with center $\QQ$.
Suppose that the endomorphism ring $O:=\End(A)$ is a maximal order in $\End(A)\otimes\QQ$. A classical result of Ohta says that the action of the absolute Galois group of $K$ on the $p$-adic Tate module $T_p(A)$ gives a representation
\[
    \rho_{A,p}:\Gal(\bar K/K)\to (O\otimes_{\ZZ}\ZZ_p)^\times
\]
for each rational prime $p$.
The representation $\rho_{A,p}$ has Frobenius traces in $\ZZ$, and it is surjective for all but finitely many $p$ \cite{ohta74}. In particular, the image of $\rho_{A,p}$ is $\GL_2(\ZZ_p)$ for all but finitely many $p$.
Observe that these properties are also enjoyed by the $p$-adic Galois representation $\rho_{E,p}$ attached to an elliptic curve $E/K$ without complex multiplication. For this reason, $A$ was called a \emph{fake elliptic curve} by Serre. We usually only have access to the representations $\{\rho_{A,p}\}_p$ through a finite set of their traces, and so there was no criterion to distinguish them from the representations attached to an elliptic curve.

Even if $\End(A)$ is not a maximal order, for every prime $p$ there is a representation
\(
    \bar\rho_{A,p}:\Gal(\bar K/K)\to \GL_2(\FF_p)
\)
such that $A[p]^{ss}\simeq \bar\rho_{A,p}\oplus \bar\rho_{A,p}$. On the other hand, given any elliptic curve $E/K$ we may consider the mod-$p$ representation given by the action of Galois on $E[p]$, which we denote by $\bar\rho_{E,p}$.

Given a system of two-dimensional residual Galois representations $\{\bar\rho_p\}_p$, we say $p$ is a \emph{non-split Cartan prime} if $\bar\rho_{p}(\Gal(\bar K/K))$ is contained in the non-split Cartan subgroup $C_{ns}(p)$ of $\GL_2(\FF_p)$ (see Section~\ref{section:non-split Cartan curves} below). We let $\mathcal C(\{\bar\rho_p\}_p)$ be the set of non-split Cartan primes for the system. If $B/K$ is and abelian variety with an associated system of representations $\{\bar\rho_{B,p}\}_p$, we let $\mathcal C_B := \mathcal C(\{\bar\rho_{B,p}\}_p)$. We say that each $p\in\mathcal C_B$ is a \emph{non-split Cartan prime for $B$}.

The main result of this paper characterizes when an abelian surface is QM or the square of an elliptic curve in terms of the set of non-split Cartan primes.

\begin{theorem}
[Theorem~\ref{theorem:fake elliptic curves}]
    Let $K$ be an imaginary quadratic field over which Conjecture~\ref{weak conj} holds. Let $A/K$ be an abelian surface such that:
    \begin{enumerate}
        \item $A$ is not the twist of a base change of an abelian surface defined over $\QQ$, 
        \item $\End(A)\otimes\QQ = \End(A_{\bar K})\otimes\QQ$, and
        \item $\End(A)\otimes\QQ$ is either $\Mat_2(\QQ)$ or a division indefinite quaternion algebra with center $\QQ$.
    \end{enumerate}
    Let $\mathcal C_A$ be the set of primes at which $A$ has residual image contained in a non-split Cartan subgroup. Then $A$ is simple and has QM if and only if $\#\mathcal C_A\geq 2$. Otherwise, $A$ is isogenous to the square of an elliptic curve.
\end{theorem}

The conjecture in the statement refers to the non-existence of non-CM $K$-rational points on the non-split Cartan modular curves $X_{ns}(p)$. 

\subsection*{Quadratic points on modular curves}

One of the implications in Theorem~\ref{theorem:fake elliptic curves} is immediate, since $\bar\rho_{A,p}$ has non-split Cartan image at every prime $p$ at which the quaternion algebra ramifies. The converse is proved by observing that elliptic curves tend to have few non-split Cartan primes. 

Let $K$ be an imaginary quadratic field and let $E$ be an elliptic curve. Let $\mathcal C_E = \mathcal C(\{\bar\rho_{E,p}\}_p)$ be the set of non-split Cartan primes for $E$. For every $p\in\mathcal C_E$, we can put a level structure on $E$ so that it defines a $K$-rational point $P_E$ on the non-split Cartan modular curve $X_{ns}(p)$. 
For the primes $p=7,11,13$, Siksek and Michaud-Jacobs showed that every quadratic point on $X_{ns}(p)$ has rational $j$-invariant \cite{michaudjacobs22}. This is not the case for smaller primes, since the curves $X_{ns}(2)$, $X_{ns}(3)$ and $X_{ns}(5)$ all have genus 0, and so there are infinitely many elliptic curves with non-rational quadratic $j$-invariant and mod-$p$ image contained in $C_{ns}(p)$ for $p=2,3,5$.

We circumvent this by assuming that $E/K$ has non-split Cartan image at two different primes $p,q\in\{2,3,5\}$. Then $E$ defines a point on the non-split Cartan curve $X_{ns}(pq)$, and we can show that $j(E)\in\QQ$. 
More precisely, the curve $X_{ns}(pq)$ has a double cover onto a certain curve $X_{ns}^+(pq)$, and we are able to show that every quadratic point on the former is a pullback of a rational point on the latter. Every such point is said to be \emph{non-exceptional} (cf. Definition~\ref{definition:exceptional-point}).
Combining our result with that of Michaud-Jacobs, we obtain the following.

\begin{theorem}[Theorem~\ref{theorem:non-exceptional XnsN}]
    Let $N\in\{6,7,10,11,13,15\}$. Every quadratic point on $X_{ns}(N)$ is non-exceptional.
\end{theorem}

The quadratic points on the genus-1 curves $X_{ns}(6)$, $X_{ns}(7)$, $X_{ns}(10)$ and $X_{ns}(11)$ are determined by using their $\mathrm{Pic}^0$ and explicit maps to other genus-0 curves. The curves $X_{ns}(13)$ and $X_{ns}(15)$ have much higher genus, and they require an application Siksek's symmetric Chabauty \cite{siksek09}.

Since the computational requirements of Chabauty grow fast with the genus, the determination of quadratic points for $X_{ns}(p)$ or $X_{ns}(pq)$ with $p\geq 17$ remains completely open. Already for the first open level, the knowledge of the set of points $X_{ns}^+(17)(\QQ)$ from \cite{bdmtv23} should aid in this regard. On the other hand, we can show the following result for imaginary quadratic fields of class number $>1$.\footnote{If $K$ has class number 1 and $p$ does not split in $K$, then $X_{ns}(p)$ will have one $K$-point for each elliptic curve with CM by $K$.}

\begin{theorem}[Corollary~\ref{corollary:finiteness of nsC p}]
    Let $K$ be an imaginary quadratic field of class number $>1$. Then there exists a constant $C_K>0$ depending only on $K$ such that, for $p>C_K$, the set $X_{ns}(p)(K)$ is empty. 
\end{theorem}

With Theorem~\ref{theorem:non-exceptional XnsN} in hand we are able to prove the following criterion.

\begin{theorem}[Theorem~\ref{theorem:exceptional set of j-invariants}]
    Let $p_1,p_2$ be distinct primes. There exists a finite set $S_{p_1,p_2}\subset \bar\QQ \setminus \QQ$ with the following property.
    If $K$ is an imaginary quadratic field and $E/K$ is an elliptic curve such that the image of $\bar\rho_{E,p_i}$ is contained in $C_{ns}(p_i)$ for $i=1,2$, then either $j(E)\in\QQ$, or
        $j(E)\in S_{p_1,p_2}$.

    Moreover, the set $S_{p_1,p_2}$ is empty if either $\{p_1,p_2\}\subset \{2,3,5\}$, or at least one of $p_1,p_2$ is in the set $\{7,11,13\}$.
\end{theorem}

\subsection*{Outline of the paper} We begin in Section~\ref{section:non-split Cartan curves} with a brief reminder on modular curves. We define the curves $X_{ns}(N)$ and $X_{ns}^+(N)$ for squarefree levels. We observe that our notation $X_{ns}^+(N)$ deviates slightly from classical sources in the literature, and instead we follow the notation used in \cite{hl2025} and the modular curves database in \cite[\href{https://beta.lmfdb.org/ModularCurve/Q/}{Modular Curves}]{lmfdb}.
We prove the result on quadratic points for levels $N=6$ and $10$ in Section~\ref{section:Xns6 10}. We then compute the quadratic points on $X_{ns}(15)$ in Section~\ref{sec:Xns15}. In Section~\ref{section:nsC primes} we discuss the implications for the set of non-split Cartan primes for an elliptic curve, and we prove Theorem~\ref{theorem:exceptional set of j-invariants}. Finally, we prove the characterization of fake elliptic curves in Section~\ref{section:fake elliptic curves}.

\subsection*{Computational tools and code availability}

Many computations in this paper were performed using Magma \cite{magma}. The code for computing the quadratic points on $X_{ns}(N)$ for $N=6,10,15$ is available at \url{https://github.com/3nr1c/nonsplit-Cartans}. In particular, the Chabauty and Mordell-Weil sieve code was adapted from \cite{githubMichaudJacobsQuadCartan}. The algorithm in \cite{zywina24} was used to obtain equations of modular curves, its implementation can be found in \cite{githubGitHubDavidzywinaOpenImage}. The packaged version of Zywina's implementation by Eran Assaf was also used, see \url{https://github.com/assaferan/OpenImage.git}.

\subsection*{Notation} Throughout the paper $K$ denotes a number field and $\bar K$ a fixed algebraic closure. The absolute Galois group is denoted by $\Gal(\bar K/K)$. For an abelian variety $A/K$, $\End(A)$ is the ring of endomorphisms of $A$ defined over $K$. If $X/K$ is a scheme and $L/K$ is an extension, we denote the base change of $X$ to $L$ by $X_K := X\times_{\operatorname{Spec} K}\operatorname{Spec} L$. The set of $L$-points of $X$ is denoted by $X(L)$. If $X/K$ is a curve, $X^{(2)}$ denotes its symmetric square, and we identify $X^{(2)}(K)$ with the set of quadratic points of $X$. Equivalence of divisors on a curve is denoted by $\sim$. If $D$ is a divisor on a curve, $\mathcal L(D)$ denotes the Riemann-Roch space of functions $f$ such that $\mathrm{div}(f)+D\geq 0$. 

\subsection*{Acknowledgements}

I thank Xevi Guitart for asking the question of whether the number of non-split Cartan primes characterizes quaternionic multiplication. I also thank Eran Assaf for pointing me to the code in \cite{githubGitHubDavidzywinaOpenImage}, which was used to obtain the equations and coverings of the modular curves considered. I thank Sam Frengley for his comments. I acknowledge funding from the project PID2022-137605NB-I00 and 2021 SGR 0146.

\section{Non-split Cartan modular curves}
\label{section:non-split Cartan curves}

We begin by briefly recalling the general setup of modular curves. Let $N$ be an integer, and let $H\subseteq \GL_2(\ZZ/N\ZZ)$. Let $L$ denote a number field. There exists a curve $Y_H$ whose $L$-rational points parametrize elliptic curves $E/L$ such that the mod-$N$ representation $\bar\rho_{E,N}:\Gal(\bar L/L)\to \GL_2(\ZZ/N\ZZ)$ has image contained in $H$. We denote the compactification of $Y_H$ by $X_H$, this is obtained by adding a finite set of cusps. The curve $X_H$ is defined over $\QQ$ as long as $\det H = (\ZZ/N\ZZ)^\times$. If $H\subseteq G\subseteq \GL_2(\ZZ/N\ZZ)$ we have a covering $X_H \to X_G$ of degree $[G:H]$. If we take $G := \GL_2(\ZZ/N\ZZ)$, then $X_G\simeq \PP^1$, and the covering $j:X_H\to \PP^1$ is called the $j$-invariant.
The points in $X_H(\CC)$ have a simple description: one takes $\Gamma_H\subseteq \SL_2(\ZZ)$ to be the lift of the group $H\cap \SL_2(\ZZ/N\ZZ)$, and then $X_H(\CC)=(\mathcal H\cup \PP^1(\QQ))/\Gamma_H$. 

In this paper we will only deal with the non-split Cartan modular curves $X_{ns}(N)$ and $X_{ns}^+(N)$ (with notation as in \cite{hl2025}), which we now define. Let $p$ be a prime number and fix an $\FF_p$-basis of $\FF_{p^2}$. The non-split Cartan subgroup modulo $p$, $C_{ns}(p)$, is the image of $\FF_{p^2}^\times$ in $\GL_2(\FF_p)$ by the action of $\FF_{p^2}^\times$ on $\FF_{p^2}$. By taking another basis, we obtain a conjugate subgroup. If we let $\varepsilon\in\FF_p^\times$ be a multiplicative generator, then we may take 
\[
C_{ns}(p) = \left\{
    \begin{pmatrix}
        x & \varepsilon y\\
        y & x
    \end{pmatrix}
    \mid
    x,y \in \FF_p,\ (x,y)\neq(0,0)
\right\}.
\]
In $\GL_2(\FF_{p^2})$, the group $C_{ns}(p)$ is conjugate to the subgroup consisting of the matrices $\left(
\begin{smallmatrix}
    \alpha&0\\
    0&\alpha^p
\end{smallmatrix}
\right)$, for all $\alpha\in\FF_{p^2}^\times$.
We denote by $C_{ns}^+(p)$ the normalizer of $C_{ns}(p)$ in $\GL_2(\FF_p)$. We have $[C_{ns}^+(p):C_{ns}(p)]=2$, and $C_{ns}^+(p)$ is generated by $C_{ns}^+(p)$ and the matrix $\left(\begin{smallmatrix}
    1 & 0\\
    0 & -1
\end{smallmatrix}\right)$.

Let $N$ be a composite squarefree integer. The non-split Cartan group of level $N$, denoted $C_{ns}(N)$, is defined as the image of $\prod_{p\mid N}C_{ns}(p)$ via the isomorphism
\[
    \prod_{p\mid N}\GL_2(\FF_p)\simeq \GL_2(\ZZ/N\ZZ).
\]
\begin{definition}
    For a squarefree integer $N\geq 2$, the non-split Cartan modular curve $X_{ns}(N)$ is the curve $X_H$ defined by the group $H=C_{ns}(N)$.
\end{definition}

The curve $X_{ns}(N)$ is always defined over $\QQ$, and it has no non-cuspidal points defined over any field with a real embedding. The reason is that the action of complex conjugation corresponds through $\bar\rho_{E,p}$ to a matrix with eigenvalues $1$ and $-1$, which is not in $C_{ns}(p)$. Over a field $L$ containing an imaginary quadratic field $K$ of class number one, then $X_{ns}(p)(L)$ has a CM, non-cuspidal point if $p$ is inert or ramified in $K$.

We denote by $C_{ns}^*(N)$ be the normalizer of $C_{ns}(N)$ in $\GL_2(\ZZ/N\ZZ)$. We have that 
\(
    [C_{ns}^*(N):C_{ns}(N)]=2^{\omega(N)},
\)
where $\omega(N)$ is the number of primes dividing $N$. The group $C_{ns}^*(N)$ is isomorphic to the product $\prod_{p\mid N}C_{ns}^+(p)$, we have
\[
    C_{ns}^*(N)/C_{ns}(N) \simeq (\ZZ/2\ZZ)^{\omega(N)}.
\]
Every nonzero element of this quotient generates a group $H\subset C_{ns}^*(N)$ with $[H:C_{ns}(N)]=2$, and so there are $2^{\omega(N)}-1$ such subgroups. Each $H$ gives us a modular curve $X_H$ with a degree-2 map
\(
    \varrho_H:X_{ns}(N)\to X_H.
\)
Observe that, when $N=pq$ is the product of two primes, we obtain three degree-2 covers. Given a point $P\in X_H(\QQ)$, every point $Q$ in the preimage $\varrho^{-1}(P)$ is defined over a quadratic number field.

\begin{definition}\label{definition:exceptional-point}
    Let $P\in X_{ns}(N)(\bar\QQ)$ be a quadratic point. We say $P$ is\\\emph{non-exceptional} if
    \[
        P\in \bigcup_{[H:C_{ns}(N)]=2}\varrho_H^{-1}(X_H(\QQ)).
    \]
    Otherwise, we say $P$ is \emph{exceptional}.
\end{definition}

We distinguish the following curve $X_H$ doubly covered by $X_{ns}(N)$. Let $W$ be a matrix in $C_{ns}^*(N)$ which restricts to an element in $C_{ns}^+(p)\setminus C_{ns}(p)$ for each $p\mid N$.
\begin{definition}
    The non-split Cartan plus modular curve $X_{ns}^+(N)$ is the curve $X_H$ defined by the group $H=C_{ns}^+(N) := \langle C_{ns}(N), W\rangle$. 
\end{definition}
We denote by $w:X_{ns}(N)\to X_{ns}(N)$ the involution induced by $W$, and we denote by $\varrho_w:X_{ns}(N)\to X_{ns}^+(N)$ the corresponding double cover.

The following lemma lists the smallest-genus curves $X_{ns}(N)$ for squarefree $N$. Recall that a curve $X$ is said to be hyperelliptic (resp. bielliptic) if there exists a double cover $X\to \PP^1$ (resp. $X\to E$ for some elliptic curve $E$).

\begin{lemma}\label{lemma:genera of Xns}
    \hfill
    \begin{enumerate}[(a)]
        \item The curves $X_{ns}(2)$, $X_{ns}(3)$ and $X_{ns}(5)$ have genus 0.
        \item 
    The non-split Cartan modular curves of squarefree level and genus 1 are $X_{ns}(6)$, $X_{ns}(7)$ and $X_{ns}(10)$.
    \footnote{There is in fact only one more non-split Cartan modular curve of genus 1, namely $X_{ns}(8)$. We have stated the lemma for squarefree level since a slightly more general definition is needed for the Cartan group of level $8$.}
        \item For every squarefree integer $N\geq 11$, the curve $X_{ns}(N)$ has genus at least 4. 
        \item For every prime $p\geq 13$, the curve $X_{ns}(p)$ is neither hyperelliptic nor bielliptic.
    \end{enumerate}
\end{lemma}
\begin{proof}
    Let $N$ be a squarefree integer. From \cite[Table~1]{hl2025} we have the genus formula
    \begin{align*}
        g(X_{ns}(N)) = 1 + \frac{(N-6)\varphi(N)}{12} - \frac{\prod_{p\mid N}\left(1-\left(\frac{-1}{p}\right)\right)}{4} - \frac{\varepsilon_3}{3},
    \end{align*}
    where $\varepsilon_3 = 2^{\omega(N)}$ if $p\equiv 2\pmod 3$ for all $p\mid N$, and $\varepsilon_3=0$ otherwise. It is easy to check that $X_{ns}(6)$, $X_{ns}(7)$ and $X_{ns}(10)$ have all genus 1. a short computation shows that for squarefree $N\geq 11$, $X_{ns}(N)$ has genus at least 4.

    We now show (d). Let $p$ be a prime. The existence of a double cover of $X_{ns}(p)$ onto some curve $C$ implies that there must be an involution $w$ on $X_{ns}(p)$. For $p\geq 13$, \cite[Theorem~5.11]{dlm22} gives that $\Aut(X_{ns}(p))=\ZZ/2\ZZ$. Hence the only double cover from $X_{ns}(p)$ is the one given by the normalizing involution, $\varrho:X_{ns}(p)\to X_{ns}^+(p)$. The genus formula for $X_{ns}^+(p)$ asserts that this curve has genus $\geq 3$ for $p\geq 13$, and therefore $X_{ns}(p)$ is neither hyperelliptic nor bielliptic.
\end{proof}

\section{Exceptional quadratic points}
\label{section:Xns6 10}
Our goal is to prove the following result.

\begin{theorem}\label{theorem:non-exceptional XnsN}
    Let $N\in\{6,7,10,11,13,15\}$. Every quadratic point on $X_{ns}(N)$ is non-exceptional.
\end{theorem}

The cases where $N$ is prime were done in \cite{michaudjacobs22}. We prove the composite cases in a separate section for each curve. We begin by giving an outline of the proof in \S\ref{subsection:strategy}, which guides both cases. We show that all quadratic points are non-exceptional when $N=6$ in \S\ref{subsection:Xns6}, and when $N=10$ in \S\ref{subsection:Xns10}. The case $N=15$ requires more work, and relies on the computation of the set $X_{ns}^+(15)(\QQ)$ and the technique of Siksek's symmetric Chabauty.

\subsection{Strategy for the genus 1 cases}
\label{subsection:strategy}
In \cite[\S3]{michaudjacobs22}, it was shown that the genus 1 curve $X_{ns}(7)$ has no exceptional quadratic points. Our method to treat the genus 1 curves $X_{ns}(6)$ and $X_{ns}(10)$ is similar to the one used in \emph{loc. cit.}, we first give an outline of the steps. Let $N=6$ or $10$.
\begin{enumerate}
    \item We find a modular curve $C$ of genus 0 with no points over $\QQ$, together with a morphism $\phi:X_{ns}(N)\to C$.
    \item In both cases, we show that $\Jac(X_{ns}(N))(\QQ)\simeq\ZZ/2\ZZ$. We identify a quadratic field $K/\QQ$ and a pair of points $T_1,T_2$ such that the divisor $T_1-T_2$ represents the nontrivial element of the Jacobian. 
    \item We use the points on $X_{ns}(N)(K)$ and the curve $C$ to prove that 
    \[\Pic^0(X_{ns}(N)/\QQ)=0.\]
    \item We identify explicitly the morphism $\varrho_w:X_{ns}(N)\to X_{ns}^+(N)$.
    \item We use the fact that $\Pic^0=0$ and the morphism $\rho$ to show that every quadratic point on $X_{ns}(N)$ is non-exceptional.
\end{enumerate}

\subsection{\texorpdfstring{$X_{ns}(6)$}{Xns(6)}}
\label{subsection:Xns6}
In this section we prove that all quadratic points on $X_{ns}(6)$ are non-exceptional. As listed in \cite[\href{https://beta.lmfdb.org/ModularCurve/Q/6.12.1.a.1/}{Modular Curve $X_{ns}(6)$}]{lmfdb}, this curve can be given by an intersection of quadrics in $\PP^3$:
\begin{equation}
\label{eq:Xns6}
X_{ns}(6):\begin{cases}
    0=-18xw + 3y^2 + z^2 + zw + w^2,\\
    0=36x^2 + 3xw - y^2.
    \end{cases}
\end{equation}
We begin by looking at some features of the curve $X_{ns}(6)$ explicitly.
The following six points on $X_{ns}(6)(\QQ(\sqrt{-3}))$ are easily found by hand:
\begin{align*}
    P = (1:6:6\sqrt{-3}:0),\quad
    &P' = (1:6:-6\sqrt{-3}:0),\\
    Q = (1:-6:6\sqrt{-3}:0),\quad
    &Q' = (1:-6:-6\sqrt{-3}:0),\\
    R = (0:0:\frac{-1+\sqrt{-3}}{2}:1),\quad
    &R' = (0:0:\frac{-1-\sqrt{-3}}{2}:1).
\end{align*}
It is easily checked that $R$ and $R'$ are the only points on this model of $X_{ns}(6)$ such that $x=0$, while $P,P',Q,Q'$ are the only points such that $w=0$.
By using the \texttt{Jacobian} function in Magma, we find that $\Jac(X_{ns}(6))$ is the elliptic curve
\[
    E:y^2 = x^3 + 1/144x^2 - 1/15552x - 1/2239488.
\]
As a sanity check, we let $P$ be the point at infinity, and we use Magma's function \texttt{EllipticCurve(C, P)} to find that a Weierstrass equation for $X_{ns}(6)_{\QQ(\sqrt{-3})}$ is
\[
    E':s^2 = r^3 + \frac{9\sqrt{-3}-9}{8}r^2 + \frac{-3\sqrt{-3}-3}{8}r + \frac{1}{16},
\]
and we check that $E'\simeq E_{\QQ(\sqrt{-3})}$.
With these models we check that $P,P',Q,Q',R$ and $R'$ are all the points in $X_{ns}(6)(\QQ(\sqrt{-3}))$. 

As explained in Section~\ref{section:non-split Cartan curves}, the curve $X_{ns}(6)$ has three quadratic subcovers. These are computed in \emph{loc. cit.} to be the following:
\begin{itemize}
    \item The genus 0 curve $X_{ns}(3)$, which has no rational points.
    \item The genus 0 curve $X_{ns}^+(6)$, which is isomorphic to $\PP^1$.
    \item An elliptic curve $E_1$ of rank 0 over $\QQ$.
\end{itemize}

The $j$-invariant map $X_{ns}(6)\to X(1)$ is described by
\begin{align*}
j = 
2^4\cdot 3^4\cdot
\frac{w^3}{216x^2w-36xy^2+9xw^2-3wy^2}.
\end{align*}
The point is that $X_{ns}(6)$ has an obvious degree-2 morphism to the genus 0 curve $36x^2+3xw-y^2=0$ (from the expression \eqref{eq:Xns6}), namely
\begin{align*}
    \varrho:(x:y:z:w) \mapsto (x:y:w).
\end{align*}
The second expression for the $j$-invariant shows that $j$ factors through $\varrho$. We observe that the curve $36x^2+3xw-y^2=0$ above is isomorphic to $\PP^1$: it has the rational point $(x:y:w)=(0:0:1)$. As $X_{ns}^+(6)$ is the only genus 0 quadratic subcover of $X_{ns}(6)$ which is isomorphic to $\PP^1$, this shows that $\varrho$ is the quadratic cover $X_{ns}(6)\to X_{ns}^+(6)$ up to an automorphism of $\PP^1$. Even more, we have that $\varrho(R)=\varrho(R')=(0:0:1)$, hence the points $R$ and $R'$ are non-exceptional. Similarly, we see that the points $P,P',Q$ and $Q'$ are also non-exceptional.

On the other hand, we can also consider the genus 0 curve 
\[C:108x^2+9xw+z^2+zw+w^2=0,\] 
obtained by taking a linear combination of the equations of $X_{ns}(6)$ to eliminate the variable $y$. The change of variables $X=x-w/24,Z=z-w/2,W=w$, shows that $C$ is equivalent to the quadratic form $108X^2+Z^2+\frac{9}{16}W^2$, which has no rational points. Hence $C(\QQ)=\emptyset$. On the other hand, we have a degree-2 morphism
\begin{align*}
    \phi:X_{ns}(6) &\to C\\
    (x:y:z:w) &\mapsto (x:z:w).
\end{align*}
We are ready to prove the following result.

\begin{proposition}
\label{proposition:Pic0 Xns6 trivial}
    The group $\Pic^0(X_{ns}(6)/\QQ)$ is trivial.
\end{proposition}
\begin{proof}
    We use the elliptic curve $E=\Jac(X_{ns}(6))$ as computed above.
    The curve $E$ has Mordell-Weil group $E(\QQ)\simeq\ZZ/2\ZZ$. We have an injective group homomorphism $\Pic^0(X_{ns}(6)/\QQ)\hookrightarrow \Jac(X_{ns}(6))(\QQ)$, and hence we need to discard the possibility that $\Pic^0(X_{ns}(6)/\QQ)$ is isomorphic to $\ZZ/2\ZZ$.  
    
    If we look at the affine patch $w=1$, the curve $X_{ns}(6)$ has equations
    \[
        \begin{cases}
            0=-18x+3y^2+z^2+z+1,\\
            0=36x^2+2x-y^2.
        \end{cases}
    \]
    Recall that $\Jac(X_{ns}(6))(\QQ)=\Pic^0(X_{ns}(6)_{\bar\QQ})^{\Gal(\bar\QQ/\QQ)}$. This implies that the nontrivial element $D$ in $\Jac(X_{ns}(6))(\QQ)$ must be represented by a two-torsion element fixed by the Galois action up to linear equivalence. 
    
    Consider the divisor
    \(
        D:=R-R',
    \)
    satisfying $2D=\on{div}(g)$ with
    \[
        g = \frac{72x-2\sqrt{-3}z-\sqrt{-3}-3}{z+\frac{\sqrt{-3}+1}{2}}.
    \]
    We check that $D$ is not principal (using e.g. Magma's \texttt{IsPrincipal}).
    If we let $\sigma\in \Gal(\bar\QQ/\QQ)$ be any element not fixing $\QQ(\sqrt{-3})$, we see that ${}^\sigma D=-D$. Therefore $D\sim {}^\sigma D$, and $D$ is indeed an element in $\Jac(X_{ns}(6))(\QQ)$ of order 2.
    
    We will now show that no divisor in $\Pic^0(X_{ns}(6)/\QQ)$ is linearly equivalent to $D$. Suppose for a contradiction that $D\sim D'$ with $D'$ a rational divisor of degree 0. If this were the case, by Riemann-Roch, we would have $D'+R+R'\sim S+S'$, where $S$ is a quadratic point and $S'$ is its Galois conjugate. Hence we have
    \[
        S+S'-2R\sim D'-(R-R')\sim D'-D \sim 0.
    \]
    Now we have $R\in X_{ns}(6)(\QQ(\sqrt{-3}))$, so there must exist a function $f\in \QQ(\sqrt{-3})(X_{ns}(6))$ with $\on{div}(f)=S+S'-2R$. In particular, $f\in\mathcal L(2R)$ and $f(S)=f(S')=0$. A basis for $\mathcal L(2R)$ is 
    \[
        \left\{1,
        \frac{-108x + (6\sqrt{-3} + 6)z + 21}{z + \frac{1-\sqrt{-3}}{2}}
        \right\}.
    \]
    Hence we may normalize $f$ by a constant so that there is some $\alpha\in\QQ(\sqrt{-3})$ such that
    \[
        f = \alpha + \frac{-108x + (6\sqrt{-3} + 6)z + 21}{z + \frac{1-\sqrt{-3}}{2}}.
    \]
    It is computed then that the vanishing locus of $f$ on $X_{ns}(6)$ satisfies
    \begin{align*}
        x &= \frac{-\sqrt{-3}\alpha - 6\sqrt{-3} + 27}
        {\alpha^2 + (12\sqrt{-3} + 12)\alpha + (72\sqrt{-3} + 36)},\\
        z &= \frac{\frac{(\sqrt{-3} - 1)\alpha^2}{2} - 33\alpha - 180\sqrt{-3} - 180}
    {\alpha^2 + (12\sqrt{-3} + 12)\alpha + (72\sqrt{-3} + 36)}.
    \end{align*}
    In particular we have $x(S)=x(S')$ and $z(S)=z(S')$. Since $S$ and $S'$ are Galois conjugates, we obtain $x(S),z(S)\in\QQ$. It follows that $\phi(S)=(x(S):z(S):1)$ defines a point on $C(\QQ)$, which is a contradicition, since $C(\QQ)=\emptyset$. 
    This shows that $\Pic^0(X_{ns}(6)/\QQ)=0$.
\end{proof}

Finally, we can prove the non-exceptionality of quadratic points on $X_{ns}(6)$. 

\begin{proposition}
\label{proposition:non-exceptional points Xns6}
    Let $S$ be a quadratic point on $X_{ns}(6)$. Then $S$ is a pullback of a rational point on $X_{ns}^+(6)$.
\end{proposition}
\begin{proof}
    We will work on the affine patch $w=1$ (we already know that letting $w=0$ yields the points $P,P',Q,Q'$, which are non-exceptional).    
    Suppose $S$ is a quadratic point on $X_{ns}(6)$ and let $S'$ be its Galois conjugate. Since this curve has trivial $\Pic^0$, we have
    \[
        S+S'-(P+P')\sim S+S'-(Q+Q')\sim 0,
    \]
    and so there exist two functions $f,g$ satisfying $\on{div}(f)=S+S'-(P+P')$ and $\on{div}(g)=S+S'-(Q+Q')$. In particular, we have $f\in\mathcal L(P+P')$ and $g\in\mathcal L(Q+Q')$. These spaces have $\QQ$-bases
    \begin{align*}
        \mathcal L(P+P')&=\langle 1,6x+y\rangle\\
        \mathcal L(Q+Q')&=\langle 1,6x-y\rangle.
    \end{align*}
    It follows that, up to normalizing by a constant,
    \begin{align*}
        f = 6x+y-\alpha,\quad
        g = 6x-y-\beta,
    \end{align*}
    for some $\alpha,\beta\in\QQ$. Since $f(S)=g(S)=0$, we necessarily have $x(S)$ and $y(S)\in\QQ$. Hence $\varrho_w(S)=(x(S):y(S):1)$ is a point on $X_{ns}^+(6)$, and in particular $S$ is non-exceptional.
\end{proof}

\subsection{\texorpdfstring{$X_{ns}(10)$}{Xns(10)}}
\label{subsection:Xns10}
We now prove that all quadratic points on $X_{ns}(10)$ are non-exceptional. From \cite[\href{https://beta.lmfdb.org/ModularCurve/Q/10.40.1.a.1/}{Modular Curve $X_{ns}(10)$}]{lmfdb}, this genus 1 curve  has a model in $\PP^4$ cut out by the equations
\begin{equation}\label{eq:Xns10}
X_{ns}(10):\begin{cases}
    0=x^2 + y^2 - yw + z^2 - zw + w^2,\\
    0=2x^2 - y^2 + yz + yw - z^2.
\end{cases}
\end{equation}

Using the \texttt{Jacobian} command in Magma, we find that $\Jac(X_{ns}(10))$ is the elliptic curve
\[
E:v^2 = u^3 + 1/5u^2 - 1/25u.
\]
Its Mordell-Weil group is $E(\QQ)\simeq\ZZ/2\ZZ$.
By pulling back the rational points $(X:Z:W)=(1:-3:1)$ and $(-1:-3:1)$ on $C'$, we obtain the following four points on $X_{ns}(10)(\QQ(\sqrt{-3}))$ in coordinates $(X:Y:Z:W)$:
\begin{align*}
    P=(1:2+\sqrt{-3}:-3:1),\quad
    P'=(1:2-\sqrt{-3}:-3:1)\\
    R=(-1:2+\sqrt{-3}:-3:1),\quad
    R'=(-1:2-\sqrt{-3}:-3:1).
\end{align*}
This makes $X_{ns}(10)_{\QQ(\sqrt{-3})}$ an elliptic curve. Choosing $P$ as its origin, we find that it has a Weierstrass model
\[
    E':s^2 = r^3 + \frac{480\sqrt{-3}-120}{49}r^2 + \frac{7200000\sqrt{-3} - 5720000}{117649}.
\]
Using this model, we check that $E'(\QQ(\sqrt{-3}))\simeq\ZZ/2\ZZ\times \ZZ$, and that there is an isomorphism $E\simeq E'$ over $\QQ$.

As in the previous section, we have three quadratic subcovers, namely
\begin{itemize}
    \item The genus 0 curve $X_{ns}(5)$, which has no rational points.
    \item The genus 0 curve $X_{ns}^+(10)$, which is isomorphic to $\PP^1$.
    \item An elliptic curve $E_1$ of rank 0 over $\QQ$.
\end{itemize}
From the model \eqref{eq:Xns10} we find a map 
\(
\phi:(x:y:z:w)\mapsto (y:z:w)
\)
from our curve to the conic $C:-3y^2 + yz - 3z^2 + 3yw + 2zw - 2w^2=0$, which is positive definite as a real quadratic form, and therefore has no points in $\QQ$. However, it is not clear how to obtain a map from $X_{ns}(10)$ to $\PP^1$ by considering linear combinations of the defining quadrics. Instead, if we consider the change of variables
\[
x=X,\quad
y=Y,\quad
z=Z+Y,\quad
w=W,
\]
then we obtain the model
\begin{equation}\label{eq:Xns10 new model}
X_{ns}(10):
\begin{cases}
    0=X^2 + 2Y^2 + 2YZ + Z^2 - 2YW - ZW + W^2,\\
    0=2X^2 - Y^2 - YZ - Z^2 + YW.
\end{cases}
\end{equation}
Now the genus 0 curve $C':5X^2 - Z^2 - ZW + W^2=0$ does have a rational point, e.g. $(X:Z:W)=(1:-3:1)$, and the map $\varrho_w:(X:Y:Z:W)\mapsto (X:Z:W)$ is a double cover of $C'\simeq \PP^1$. With the coordinates $X,Y,Z,W$, the $j$-invariant map to $X(1)$ is given by the function $-5^4f/g$, where 
\begin{align*}
    f &=-256 Z^{10} + 1152 Z^{9} W - 2304 Z^{8} W^{2} + 2840 Z^{7} W^{3} - 2511 Z^{6} W^{4} \\ 
    &\phantom{=}\ + 1686 Z^{5} W^{5} - 869 Z^{4} W^{6} + 348 Z^{3} W^{7} - 105 Z^{2} W^{8} + 22 Z W^{9} - 3 W^{10},\\
    g & =-Z^{10} - 5 Z^{9} W - 5 Z^{8} W^{2} + 10 Z^{7} W^{3} + 15 Z^{6} W^{4} - 11 Z^{5} W^{5}\\
    &\phantom{=}\ - 15 Z^{4} W^{6} + 10 Z^{3} W^{7} + 5 Z^{2} W^{8} - 5 Z W^{9} + W^{10}.
\end{align*}
As $X_{ns}^+(10)$ is the only degree 2 subcover of $j:X_{ns}(10)\to X(1)$ isomorphic to $\PP^1$, and $j$ factors through $C'$, we necessarily have $C'\simeq X_{ns}^+(10)$ over $\QQ$.

\begin{proposition}
    The group $\Pic^0(X_{ns}(10)/\QQ)$ is trivial.
\end{proposition}
\begin{proof}
    The proof is similar to that of Proposition~\ref{proposition:Pic0 Xns6 trivial}.
    We work with the model in Eq.~\eqref{eq:Xns10 new model} on the affine patch $W=1$.  We know that $\Jac(X_{ns}(10))(\QQ)\simeq\ZZ/2\ZZ$. The nontrivial divisor class is represented by $D=R'-P$. Indeed, the point $R'$ is not be linearly equivalent to $P$, and the function
    \[
    g=\frac{(6Y + (-2\sqrt{-3} + 3)Z - 3)  (2X - Z - 1)}{(Z+3)^2}
    \]
    satisfies $\on{div}(g)=2D$. Moreover,  Magma's \texttt{IsLinearlyEquivalent} reveals that we have $R'-P\sim R-P'$. 
    We have an injective homomorphism 
    \[
        \Pic^0(X_{ns}(10)/\QQ)\hookrightarrow\Jac(X_{ns}(10))(\QQ).
    \]
    We wish to show that $D$ is not in the image.
    Suppose for a contradiction that $D\sim D'$ for some rational divisor $D'$ of degree 0. Since $R+R'$ is also a rational divisor, by Riemann-Roch there exists some quadratic point $S$ (with quadratic conjugate $S'$) such that $D'+R+R'\sim S+S'$. Then we have
    \[
    S+S'-R-P\sim S+S'-R-R'-(P-R') \sim D'-D\sim 0.
    \]
    Therefore there exists some function $f\in\mathcal L(R+P)$ such that $f(S)=f(S')=0$. A basis for $\mathcal L(R+P)$ is 
    \[
    \langle 1,\frac{Y - (2-\sqrt{-3})}{Z + 3}\rangle.
    \]
    Hence, after rescaling there exists some $\alpha\in\QQ$ such that $f=\alpha + (Y - (2-\sqrt{-3})/(Z + 3)$. A computation shows that the $Y$ and $Z$-coordinates of the zeros of $f$ satisfy
    \begin{align*}
    Y &= \frac{(\sqrt{-3} + 2)\alpha^2 + 4\alpha + (-3\sqrt{-3} + 
        6)/5}{\alpha^2 + \alpha + 3/5},\\
    Z &= \frac{-3\alpha^2 + (2\sqrt{-3} - 3)\alpha + 
        (5\sqrt{-3} + 1)/5}{\alpha^2 + \alpha + 3/5}.
    \end{align*}
    Since $S$ and $S'$ are conjugates, we know that $Y(S)=Y(S')$ and $Z(S)=Z(S')$, and hence these coordinates lie in $\QQ$. But this produces a point $\phi(S)=(Y(S),Y(S)+Z(S),1)\in C(\QQ)$, which is a contradiction, since $C$ has no rational points. Hence $\Pic^0(X_{ns}(10)/\QQ)$ is trivial.
\end{proof}

We now show that $X_{ns}(10)$ only has non-exceptional quadratic points.

\begin{proposition}
    Let $S$ be a quadratic point on $X_{ns}(10)$. Then $S$ is a pullback of a rational point on $X_{ns}^+(10)$.
\end{proposition}
\begin{proof}
    We work again on the patch $W=1$. Let $S$ be a quadratic point on $X_{ns}(10)$. Since $\Pic^0(X_{ns}(10)/\QQ)=0$, we know that $S+S'-R-R'$ is a principal divisor, and so there exists some $h\in\mathcal L(R+R')$ such that $h(S)=h(S')=0$. Now a basis for this Riemann-Roch space is $\{ 1,(x - 1)/(z + 3)\}$, and so (up to a constant) we have $h=\alpha+(x-1)/(z+3)$ for some $\alpha\in\QQ$. A computation shows that an expression for a zero of $h$ is
    \[
        \left(\frac{-\alpha^2 - \alpha - 1/5}{\alpha^2 - 1/5} : \beta : \frac{-3\alpha^2 - 
        2\alpha - 2/5}{\alpha^2 - 1/5} : 1\right),
    \]
    for some quadratic $\beta\in\bar\QQ$. If follows that the $X$ and $Z$-coordinates of the points $S$ and $S'$ are rational. Therefore $\varrho_w(S)=(X(S):Z(S):1)\in X_{ns}^+(10)(\QQ)$, as claimed.
\end{proof}

\subsection{\texorpdfstring{$X_{ns}(15)$}{Xns(15)}}
\label{sec:Xns15}

In this section we show that the genus-7 curve $X_{ns}(15)$ has no exceptional quadratic points. We will do this using symmetric Chabauty \cite{siksek09}, we begin by giving the necessary ingredients.

Using the algorithm described in \cite[\S5]{zywina24} (as implemented in \cite{githubGitHubDavidzywinaOpenImage}), we find a model for the curve $X_{ns}(15)$ given by 10 equations in $\PP^6$ (see Appendix~\ref{appendix:Xns15}). These equations have been chosen so that the involution $w:X_{ns}(15)\to X_{ns}(15)$ inducing the double cover $\varrho:X_{ns}(15)\to X_{ns}^+(15)$ is given by
\[
    w(x_1:x_2:x_3:x_4:x_5:x_6:x_7) = 
        (x_1:x_2:-x_3:-x_4:-x_5:-x_6:-x_7).
\]
Using Magma's \texttt{CurveQuotient} we obtain the model $y^2+y = x^6+3x^5-5x^3+3x$ of $X_{ns}^+(15)$. The set of rational points on this curve was computed in \cite{frengley23}.

\begin{proposition}[Frengley]
    The curve $X_{ns}^+(15)$ has exactly 14 points over $\QQ$. In the model $y^2+y=x^6+3x^5-5x^3+3x$, they are given by the set
    \begin{align*}
    \{
    (1,-2),(-2,-2),(-1,-1),(0,-1),(1,1),(-2,1),(1,-1),
    (-1,0),\\(0,0),(3,36),(-4,36)(3,-37),(-4,-37),
    \infty_+,\infty_-
    \}
    \end{align*}
    where $\infty_+ = (1:1:0)$ and $\infty_- = (1:-1:0)$ in the weighted projectivization of the model.
\end{proposition}

We can easily take pullbacks of the points in $X_{ns}^+(15)(\QQ)$ to obtain a set of known non-exceptional points on $X_{ns}(15)^{(2)}(\QQ)$. They are listed in Table~\ref{table:Xns15-known}.
We next describe the Mordell-Weil groups of the jacobians of our curves.

\begin{lemma}\label{lemma:ranks-15}
    Let $J:=\Jac(X_{ns}(15))$ and $J^+:=\Jac(X_{ns}^+(15))$. 
    \begin{enumerate}
        \item $J^+(\QQ)\simeq \ZZ^2$, with generators $\Delta_1 := P_4+P_9 - P_7 - P_8$ and $\Delta_2 := P_5 - P_7$.
        \item $J(\QQ)\simeq \ZZ^2 \oplus T$, where the order of $T$ is only divisible by 2 and 3. 
        \item The divisors $D_1 := \varrho^*(\Delta_1)$ and $D_2 := \varrho^*(\Delta_2)$ generate a subgroup of $J(\QQ)$ index $I\mid 2^5\cdot 3^4$.
    \end{enumerate}
\end{lemma}
\begin{proof}
    For the jacobian of the curve $X_{ns}^+(15)(\QQ)$, we use Magma's \texttt{MordellWeilGroup} find that $J^+(\QQ)\simeq \ZZ\oplus \ZZ$ with generators $\Delta_1$ and $\Delta_2$. For the rank of $J$, we use the isogeny
    \[
        J_{ns}(15) \sim J_0^{new}(15^2)
    \]
    from \cite[Theorem~3.8]{dlm22}. The space $S_2^{new}(\Gamma_0(15^2))$ contains five normalized eigenforms with rational coefficients, plus a pair of conjugate eigenforms with coefficients in $\QQ(\sqrt{5})$\footnote{The relevant data can be checked at \url{https://www.lmfdb.org/ModularForm/GL2/Q/holomorphic/?level=15\%5E2&weight=2&char\_label=225.a}}. All of the forms have analytic rank zero, except for two forms corresponding to elliptic curves that have rank one. By \cite[Theorem~0.3]{kl89} we have that the free part of $J(\QQ)$ has rank 2. To bound the torsion subgroup $T$, we compute that
    \begin{align*}
        \#J(\FF_{11}) = 2^{13} \cdot 3^4 \cdot 7^2,\quad
        \#J(\FF_{13}) = 2^4 \cdot 3^4 \cdot 5^1 \cdot 19^1.
    \end{align*}
    Using the injectivity of torsion points, we find that $\#T$ divides $2^4\cdot 3^4$.

    To prove (3), we observe that since $\deg\varrho = 2$, we have $\varrho_*D_i = \varrho_*\varrho^*\Delta_i = 2\Delta_i$. From (1) we know that $J^+(\QQ)=\langle D_1,D_2\rangle$, and therefore $\varrho_*\langle \Delta_1, \Delta_2 \rangle = 2J^+(\QQ)$. For every $X\in J(\QQ)$, we have
    \[
        \varrho_*(2X) = 2\varrho_*X \in 2J^+(\QQ) \subseteq \varrho_*\langle \Delta_1, \Delta_2\rangle.
    \]
    Hence there is some $Y\in \langle \Delta_1,\Delta_2\rangle$ such that $\varrho_*(2X) = \varrho_*(Y)$, and we obtain $\varrho(2X-Y)=0$. Since $J(\QQ)$ and $J^+(\QQ)$ both have rank 2, necessarily $2X-Y\in T$. Therefore we have that $2^5\cdot 3^4\cdot J(\QQ) \subseteq \langle \Delta_1,\Delta_2\rangle$, which implies that the index $I=[J(\QQ):\langle \Delta_1,\Delta_2\rangle]$ is divides $2^5\cdot 3^4$.
\end{proof}

We now have all the necessary information on $X_{ns}^+(15)$ and $X_{ns}(15)$ to find the quadratic points.

\begin{theorem}\label{theorem:Xns15}
    All quadratic points on $X_{ns}(15)$ are non-exceptional.
\end{theorem}
\begin{proof}
    We perform symmetric Chabauty using the double cover $\varrho:X_{ns}(15)\to X_{ns}^+(15)$ induced by the involution $w:X_{ns}(15)\to X_{ns}(15)$. As in Lemma~\ref{lemma:ranks-15}, we let $X=X_{ns}(15)$ and $J=\Jac(X_{ns}(15))$, and we denote the symmetric square of $X$ by $X^{(2)}$. We wish to find the set of rational points $X^{(2)}(\QQ)$. We let $\iota:X^{(2)}\to J$ be the Abel-Jacobi map defined by $\iota(\{P,Q\}) = [P+Q-\infty_+-\infty_-]$, and we let $\iota_p$ be the corresponding Abel-Jacobi map over $\FF_p$. We let $X^{(2)}_\text{known}$ be the set of points $\{Q_i,{}^\sigma Q_i\}$ on $X^{(2)}$ given by Table~\ref{table:Xns15-known}. We denote by $\tilde\cdot$ the reduction modulo a prime $p$ of a point in $X^{(2)}(\QQ)$ or $J(\QQ)$.

    For every prime $p\in\{7,\dots,59\}$, we begin by computing a basis of differentials for $X(\FF_p)$. Then we consider a point $\{Q_i,{}^\sigma Q_i\}\in X^{(2)}_\text{known}$. We let $t_i$ be a uniformizer for $\tilde Q_i\in X(\FF_p)$, and we find that there exists a differential $\omega\in \Omega_{X(\FF_p)}$ such that
    \[
        \frac{v}{dt_i}\big|_{t_i=0}\neq 0
    \quad\text{ and}\quad (1+w)^*v = 0.
    \]
    By Theorem~5.1 and Proposition~5.2 in \cite{michaudjacobs22}, this shows that every point $\{Q_i,{}^\sigma Q_i\}$ is alone in its mod-$p$ residue disk in $X^{(2)}(\FF_p)$.

    It remains to show that $X^{(2)}_\text{known} = X^{(2)}(\QQ)$. We do that by using the Mordell-Weil sieve as explained in \cite{michaudjacobs22}. We consider the following data.
    \begin{itemize}
        \item An integer $I\mid 2^5\cdot 3^4$ such that $I\cdot J(\QQ) \subseteq \langle \Delta_1,\Delta_2\rangle$ (cf. Lemma~\ref{lemma:ranks-15}).
        \item The parameter $M = (5\cdot 7\cdot 11\cdot 13\cdot 17\cdot 19\cdot 23\cdot 29\cdot 31)^{10}$.
        \item The set of primes $S=\{7, 13, 19, 29, 31, 37, 41, 43, 47, 53, 59\}$.
    \end{itemize}
    We check that, for every prime $p\in S$, our model of $X$ has good reduction modulo $p$. In addition, our knowledge of the prime divisors of $I$ implies that $I$ is always coprime to $\#J(\FF_p)/MJ(\FF_p)$. We let $\mu_{p,M}:J(\FF_p)\to J(\FF_p)/MJ(\FF_p)$ be the mod-$M$ projection. In addition, we let $\iota_{p,M} = \mu_{p,M}\circ \iota_p$, and for $(a,b)\in\ZZ^2$, we let $\varphi_{p,M}:=\mu_{p,M}(a\tilde \Delta_1 + b\tilde \Delta_2)$. The maps are summarized in the following diagram:

    \[\xymatrix{
    X^{(2)}_\text{known} \ar[r] & X^{(2)}(\QQ) \ar[r]^\iota \ar[d]_{\sim} & J(\QQ)\ar[d]^{\sim} & G \ar[l] & \ZZ^2 \ar[l] \ar[ddll]^{\varphi_{p,M}}\\
    & X^{(2)}(\FF_p) \ar[r]^{\iota_p} \ar[dr]_{\iota_{p,M}} & J(\FF_p) \ar[d]^{\mu_{p,M}}\\
    & & J(\FF_p) / MJ(\FF_p).
    }\]

    For each $p\in S$, we let $\mathcal T_{p,M}$ be the set of points $P\in X^{(2)}(\FF_p)$ satisfying:
    \begin{enumerate}
        \item $\iota_{p,M}(P)\in \phi_{p,M}(\ZZ^2)$, and
        \item $P \neq \tilde Q$, for all $Q\in X^{(2)}_{\text{known}}$.
    \end{enumerate}
    A standard argument using the coprimality of $I$ with $\#J(\FF_p)/MJ(\FF_p)$ shows that, if $P$ is a point in $X^{(2)}(\QQ)\setminus X^{(2)}_\text{known}$, then $\tilde P\in \mathcal T_{p,M}$ (cf. \cite[pg.~12]{michaudjacobs22}). Now we let
    \[
        \mathcal W_{p,M} := \varphi^{-1}_{p,M}(\iota_{p,M}(\mathcal T_{p,M})),
    \]
    which is a set of $\mathcal B_{p,M} := \ker(\varphi_{p,M})$-cosets. Given another prime $p'\in S$, we let $\mathcal B_{p',M} := \ker(\varphi_{p',M})$, and we let $\mathcal B_{p,p',M}:=\mathcal B_{p,M}\cap \mathcal B_{p',M}$. By seeing each coset in $\mathcal W_{p,M}$ and $\mathcal W_{p',M}$ as a union of $\mathcal B_{p,p',M}$-cosets, we may define
    \[\mathcal W_{p,p',M} := \mathcal W_{p,M} \cap \mathcal W_{p',M},\]
    which is now a set of $\mathcal B_{p,p',M}$-cosets. Performing these intersections for the remaining primes of $S$, we obtain the Mordell-Weil sieve: if the set
    \(
    \bigcap_{p\in S} \mathcal W_{p,M}
    \) is empty, then $X^{(2)}_\text{known} = X^{(2)}(\QQ)$.  

    We let $S_{\leq q} := \{p\in S\mid p\leq q\}$. The successive rounds of the Mordell-Weil sieve are detailed in Table~\ref{table:mordell-weil sieve}. For each $q\in S$, we detail the index of $\mathcal B_{S_{\leq q},M}$ and the number of elements in $\mathcal W_{S_{\leq q},M}$. The sieve ends when the latter is empty.

    \begin{table}[h]\label{table:mordell-weil sieve}
    \centering
    \small
    \begin{tabular}{|c|*{11}{c|}}
    \hline
    $p_1,\ldots,p_k$ & 7 & 7, 13 & 7, 13, 19 & $S_{\leq 29}$ & $S_{\leq 31}$ & $S_{\leq 37}$  \\ \hline
    $[\mathbb{Z}^2 : \mathcal{B}_{p_1,\ldots,p_k,M}]$ & 143 & 35321 & 6181175 & 30905875 & 216341125 & 216341125  \\ \hline
    $\#\mathcal{W}_{p_1,\ldots,p_k,M}$ & 15 & 75 & 1350 & 6450 & 3590 & 2145 \\ \hline
    \end{tabular}

    \vspace{1em}
    \begin{tabular}{|c|*{11}{c|}}
    \hline
    $p_1,\dots, p_k$ & $S_{\leq 41}$ & $S_{\leq 43}$ & $S_{\leq 47}$ & $S_{\leq 53}$ & $S_{\leq 59}$ \\ \hline
    $[\mathbb{Z}^2 : \mathcal{B}_{p_1,\ldots,p_k,M}]$  & 3677799125 & 3677799125 & 3677799125 & 3677799125 & 3677799125  \\ \hline
    $\#\mathcal{W}_{p_1,\ldots,p_k,M}$ & 7615 & 3565 & 315 & 315 & 0\\ \hline
    \end{tabular}
    \end{table}
\end{proof}

\section{Non-split Cartan primes and some uniformity results}\label{section:nsC primes}

We now use the results on low-genus non-split Cartan curves to prove the following theorem.

\begin{theorem}
\label{theorem:exceptional set of j-invariants}
    Let $p_1,p_2$ be different primes. There exists a finite set $S_{p_1,p_2}\subset \bar\QQ \setminus \QQ$ with the following property.
    If $K$ is an imaginary quadratic field and $E/K$ is an elliptic curve such that the image of $\bar\rho_{E,p_i}$ is contained in $C_{ns}(p_i)$ for $i=1,2$, then either 
        $j(E)\in\QQ$ or
        $j(E)\in S_{p_1,p_2}$.
        
     Moreover, the set $S_{p_1,p_2}$ is empty if either $\{p_1,p_2\}\subset \{2,3,5\}$, or at least one of $p_1,p_2$ is in the set $\{7,11,13\}$.
\end{theorem}
\begin{proof}
    If $\{p_1,p_2\}\subset\{2,3,5\}$, then by Theorem~\ref{theorem:non-exceptional XnsN} all quadratic points on $X_{ns}(p_1p_2)$ are non-exceptional. Hence we may set $S_{p_1,p_2}=\emptyset$. Indeed, suppose $K$ is a quadratic field and $E/K$ is an elliptic curve with non-split Cartan image modulo $p_1$ and $p_2$. By considering an appropriate $(p_1p_2)$-level structure on $E$, we obtain a point $P$ on $X_{ns}(p_1p_2)(K)$, which is the pullback of some point on $X_{ns}^+(p_1p_2)(\QQ)$. Hence $j(E)=j(P)\in\QQ$.

    Suppose without loss of generality that $p_1\not\in \{2,3,5\}$. If $p_1\in\{7,11,13\}$ then a similar reasoning yields a point on $X_{ns}(p_1)(K)$ which is non-exceptional. Hence we have $j(E)\in \QQ$, and we let $S_{p_1,p_2} := \emptyset$. 
    
    Suppose otherwise that $p_1\geq 17$. By Lemma~\ref{lemma:genera of Xns}, we know that $X_{ns}(p_1)$ is neither bielliptic nor hyperelliptic. By \cite[Proposition~1]{hs91}, $X_{ns}(p_1)$ has finitely many quadratic points. Let
    \[
        S_{p_1} := \{j(P) \mid P\in X_{ns}(p_1)(\bar\QQ) \text{ quadratic point}\} \cap (\bar\QQ\setminus \QQ).
    \]
    Then $S_{p_1}$ is a finite set, and we may similarly define the set $S_{p_2}$. Finally, we let $S_{p_1,p_2} := S_{p_1}\cap S_{p_2}$.  We again have that an elliptic curve $E/K$ yields (after considering the appropriate level structures) points $P_1\in X_{ns}(p_1)(K)$ and $P_2\in X_{ns}(p_2)(K)$, and so $j(E)=j(P_1)=j(P_2)\in S_{p_1,p_2}$.
\end{proof}

\begin{remark}
    As observed by Sam Frengley, the fact that $S_{3,5}=\emptyset$ can also be shown by using that an elliptic curve over $K$ with mod-3 image contained in $C_{ns}(3)$ and mod-5 image contained in $C_{ns}^+(5)$ has rational $j$-invariant (cf. \cite[Proposition~7.4.5]{cn23}).
\end{remark}


Computing the sets $S_{p_1,p_2}$ for larger primes is bound to require intense computational resources in order to carry out the techniques of symmetric Chabauty and the Mordell-Weil sieve. Nonetheless, it is reasonable to expect that these sets will be empty. In this spirit, we make the following conjecture.

\begin{conjecture}
\label{strong conj}
    For every prime $p\geq 7$, the curve $X_{ns}(p)$ has no exceptional quadratic points.
\end{conjecture}

At this moment we cannot provide a general argument as to why Conjecture~\ref{strong conj} should be true. However, we can give the following evidence: for a fixed imaginary quadratic field $K$ of class number $>1$ and a large prime $p$, the curve $X_{ns}(p)$ cannot have non-cuspidal $K$-points. This is a particular case of the following result.

\begin{proposition}\label{proposition:finiteness of nsC p}
    Let $K$ be a number field not containing the Hilbert class field of an imaginary quadratic number field. Then there are finitely many primes $p$ such that $X_{ns}(p)(K)$ has non-cuspidal points. Equivalently, there exists a constant $C_K>0$ such that, for $p>C_K$, no elliptic curve $E/K$ has mod-$p$ image contained in $C_{ns}(p)$. 
\end{proposition}
\begin{proof}    
    Let $p$ be a prime and suppose that $P\in X_{ns}(p)(K)$ is a noncuspidal point. Let $E/K$ be the corresponding elliptic curve, whose mod-$p$ representation $\bar\rho_{E,p}$ has image contained in $C_{ns}(p)$. We will prove that $p$ is bounded by a constant $C_K$ depending only on $K$. Without loss of generality, we may assume that $p$ is unramified in $K$ (in particular, $p>\Delta_K$).

    Let $v\mid p$ be a prime of $K$. We begin by observing that either $E$ has potentially good reduction at $v$, or $p\leq 3$. Indeed, suppose that $E$ has potentially multiplicative reduction. Let $K'$ be a finite extension of $K_v$ such that $E/K'$ has multiplicative reduction, we may take this field so that $[K':K_v]\leq 2$. Then by \cite{serre72}, the semisimplification of the restriction to the inertia group $I_{K'}$ satisfies
    \[
        \bar\rho_{E,p}|_{I_{K'}}^{ss} \simeq 1\oplus \chi_p^2,
    \]
    where $\chi_p$ is the mod-$p$ cyclotomic character (we are using that $K_v/\QQ_p$ is unramified). If $p>3$, then $\chi_p^2$ is not the trivial character. But then the image of $\bar\rho_{E,p}$ contains a matrix which cannot be in $C_{ns}(p)$, which is a contradiction.

    Suppose now $E$ has potentially good reduction at $v$. There exists a finite extension $L/K_v$ of degree 12, such that $E_L$ has good reduction. The non-split Cartan image for $\bar\rho_{E,p}$ implies there exists a character $\psi:\Gal(\bar K/K)\to \FF_{p^2}^\times$ such that $\bar\rho_{E,p}\otimes \FF_{p^2}\simeq \psi \oplus \psi^p$. If we let $c$ be the ramification degree of $L/K_v$, then by the discussion in \cite[\S~11]{serre72} and the restriction that $\bar\rho_{E,p}$ has image contained in $C_{ns}(p)$, we see that the inertia group $I_L$ of $L$ satisfies 
    \[
        \psi|_{I_L}\simeq \theta_2^d,
    \]
    where $\theta_2:I_t\to \FF_{p^2}^\times$ is the fundamental character of level 2 of the tame inertia $I_t$, and $0\leq d\leq c\leq 12$. 

    Now by \cite[Theorem~6.4]{lv14}, there exists a constant $LV_K$ such that either $p\leq LV_K$, or the equality 
    \(
        \psi^{12} = \chi_p^6
    \)
    holds (there cannot be a congruence of $\psi$ with the character coming from an elliptic curve $E'/K$ with $\End(E'_{/K})$ a CM order, as no such elliptic curve exists over $K$). Suppose that the latter is true. By restricting to $I_L$, we get an equality of inertial characters 
    \begin{equation}\label{eq:equality of characters}
        \theta_2^{12d} = \chi_p^6.
    \end{equation}
    But now $\on{ord}(\theta_2^{12d}) = \frac{p^2-1}{\gcd(p^2-1,12d)} \geq \frac{p^2-1}{\gcd(p^2-1,144)}$, while $\on{ord}(\chi_p^6) = \frac{p-1}{\gcd(p-1,6)}$. A computation shows that for all $p\geq 73$, these two orders cannot be equal.

    Putting everything together, we set $C_K:=\max\{73,{\Delta_K}, LV_K\}$. Then for all $p>C_K$, there is no elliptic curve $E/K$ with non-split Cartan image modulo $p$.
\end{proof}

\begin{corollary}
    \label{corollary:finiteness of nsC p}
    Let $K$ be an imaginary quadratic field of class number $>1$. Then there exists a constant $C_K>0$ depending only on $K$ such that, for every prime $p>C_K$, the set $X_{ns}(p)(K)$ is empty. 
\end{corollary}
\begin{proof}
    If $K$ is imaginary quadratic and has class number $>1$, then in particular it does not contain any Hilbert class fields of imaginary quadratic fields, and Proposition~\ref{proposition:finiteness of nsC p} applies. Hence there is a constant $C_K\geq 73$ such that $X_{ns}(p)(K)$ only has cuspidal points for $p>C_K$. As explained in \cite[pgs. 194-195]{serre89}, the curve $X_{ns}(p)$ has $p-1$ cusps, and they are all conjugate under the action of $\Gal(\QQ(\zeta_p)/\QQ)$. Since $p>73$, there cannot be any $K$-rational cusps, because $K$ is quadratic. Hence $X_{ns}(p)(K)$ is the empty set.
\end{proof}

\begin{remark}
    On first glance, Corollary~\ref{corollary:finiteness of nsC p} is a consequence of \cite[Theorem~12]{lv14b}, since there are no elliptic curves with CM over such a field $K$. However, the proof provided by the authors assumes GRH, while ours is unconditional, since the equality in \ref{eq:equality of characters} can only happen for finitely many primes.
\end{remark}

Finally, we make one more conjecture towards the statement Theorem~\ref{theorem:fake elliptic curves}, which is implied by Conjecture~\ref{strong conj}.
\begin{conjecture}
\label{weak conj}
    Let $K$ be an imaginary quadratic field. For every pair of primes $p\neq q$, every point $P\in X_{ns}(pq)(K)$ is non-exceptional.
\end{conjecture}

\begin{remark}
    Conjecture~\ref{weak conj} is equivalent to the set $S_{p,q}$ from Theorem~\ref{theorem:exceptional set of j-invariants} being disjoint from the field $K$.
\end{remark}

\section{Characterization of fake elliptic curves}
\label{section:fake elliptic curves}

In this section we apply our results on non-split Cartan curves to the problem of deciding whether an abelian surface has quaternionic multiplication. We begin by recalling the setting given in the introduction. 

\begin{proposition}\label{properties of QM surfaces}
    Let $K$ be an imaginary quadratic field, and let $A/K$ be an abelian surface such that $\End(A)$ is an order in an indefinite division quaternion algebra $D$ with center $\QQ$. 
    \begin{enumerate}
        \item For every rational prime $p$, there exists a representation 
        \[
            \bar\rho_{A,p}:\Gal(\bar K/K)\to\GL_2(\FF_p)
        \]
        such that $A[p]^{ss}\simeq \bar\rho_{A,p}\oplus \bar\rho_{A,p}$. The representations $\{\bar\rho_{A,p}\}_p$ form a compatible system.
        \item If the algebra $D$ ramifies at $p$, then $\bar\rho_{A,p}(\Gal(\bar K/K))\subseteq C_{ns}(p)$.
        \item In particular, we have $\#\{p\mid \bar\rho_{A,p}(\Gal(\bar K/K))\subseteq C_{ns}(p)\}\geq 2$.
    \end{enumerate}
\end{proposition}
\begin{proof}
    Suppose first that $\mathcal O = \End(A)$ is a maximal quaternion order, the $p$-adic Tate module $T_p(A)$ is a free $(\mathcal O\otimes \ZZ_p)$-module of rank 1, and hence the action of Galois on $T_p(A)$ yields a representation $\rho_{A,p}:\Gal(\bar K/K)\to(\mathcal O\otimes\ZZ_p)^\times$ (cf. \cite{ohta74}). For a prime $p$ at which $D$ splits, it is easy to check that $T_p(A)\simeq \rho_{A,p}\oplus\rho_{A,p}$. The reduction of $\rho_{A,p}$ modulo $p$ necessarily takes values in $\GL_2(\FF_p)$ for every $p$, in virtue of Wedderburn's little theorem that every central simple $\FF_p$-algebra splits. In combination with the fact that the $T_p(A)$ form a system of compatible representations, this proves the existence of $\bar\rho_{A,p}$ in the case of a maximal order.

    If $\End(A)$ is not maximal, then by \cite[Proposition~3.3]{dp13} the surface $A$ is $K$-isogenous to another surface $A'$ whose endomorphism ring is a maximal order. In particular, the representation $A'[p]$ splits as the direct sum of two copies of a representation $\bar\rho_{A',p}:\Gal(\bar K/K)\to \GL_2(\FF_p)$. Since the traces of Frobenius at almost all primes $\lambda$ of $K$ coincide for $A$ and $A'$, we have an isomorphism of semisimplifications $A[p]^{ss}\simeq A'[p]^{ss}$. This proves (1).    
    
    For (2), we may use \cite[\S4]{jordan86}. Alternatively, one can use the explicit approach in \cite[\S2.4]{schembri19thesis}. Part (3) follows from the fact that $D$ is indefinite and division, and hence it ramifies in at least two finite primes.
\end{proof}

The proposition above gives us a necessary condition for an abelian surface to have quaternionic multiplication. Namely, if we have an abelian surface $A/K$, then its $p$-torsion $A[p]$ must decompose in two equal subrepresentations, and for at least two of these primes the image of Galois needs to be contained in the non-split Cartan subgroup of $\GL_2(\FF_p)$. On the other hand, the results of the previous section allow us to say that an elliptic curve $E/K$ (and hence the surface $E^2$) has at most one non-split Cartan prime. As we  now show, this is the case unless $j(E)\in\QQ$, or $j(E)$ belongs to a finite set of exceptions.

We are ready to prove the characterization for fake elliptic curves.

\begin{theorem}\label{theorem:fake elliptic curves}
    Let $K$ be an imaginary quadratic field over which Conjecture~\ref{weak conj} holds. Let $A/K$ be an abelian surface such that:
    \begin{enumerate}
        \item $A$ is not the twist of a base change of an abelian surface defined over $\QQ$, 
        \item $\End(A)\otimes\QQ=\End(A_{\bar K})\otimes\QQ$, and
        \item $\End(A)\otimes\QQ$ is either $\Mat_2(\QQ)$ or a division indefinite division quaternion algebra with center $\QQ$.
    \end{enumerate}
    Let $\mathcal C_A$ be the set of primes at which $A$ has residual image contained in a non-split Cartan subgroup. Then $A$ is simple and has QM if and only if $\#\mathcal C_A\geq 2$. Otherwise, $A$ is isogenous to the square of an elliptic curve.
\end{theorem}

\begin{proof}
    Recall that $A$ is simple if and only if $\End(A)\otimes\QQ$ is a division algebra. Condition (3) ensures that $A$ has a system of residual Galois representations $\{\bar\rho_{A,p}:\Gal(\bar K/K)\to \GL_2(\FF_p)\}_p$ (either because $A$ is isogenous to the square of an elliptic curve, or because it has QM). By \cite{serre72} and \cite{ohta74}, condition (2) ensures that $\bar\rho_{A,p}(\Gal(\bar K/K))=\GL_2(\FF_p)$ for all but finitely many $p$, and in particular $\mathcal C_A$ is finite. 
    
    If $\End(A)\otimes\QQ$ is a division quaternion algebra, then by Proposition~\ref{properties of QM surfaces} we know that $\#\mathcal C_A\geq 2$. 
    Conversely, suppose that $\#\mathcal C_A\geq 2$. Suppose for a contradiction that $A\sim E^2$, where $E/K$ is some elliptic curve, and let $p_1,p_2\in \mathcal C_A$ be distinct primes. Let $S_{p_1,p_2}$ be the set from Theorem~\ref{theorem:exceptional set of j-invariants}. Since we are assuming that Conjecture~\ref{weak conj} is true over $K$, we know that the intersection $S_{p_1,p_2}\cap K$ is empty. This implies that $j(E)\in \QQ$, which cannot be the case, since we are assuming that no twist of $A$ can be defined over $\QQ$. 
    Hence we have a contradiction, and $A$ must be a simple surface. The only remaining possibility is that $\End(A)\otimes\QQ$ is a division quaternion algebra.
\end{proof}

\nocite{magma}
\nocite{githubFloritNonsplitCartans}

\bibliographystyle{alpha}
\bibliography{references}

@misc{hl2025,
      title={Genus formulas for families of modular curves}, 
      author={Asimina S. Hamakiotes and Jun Bo Lau},
      year={2025},
      eprint={2501.10883},
      archivePrefix={arXiv},
      primaryClass={math.NT},
      url={https://arxiv.org/abs/2501.10883}, 
}

@article{michaudjacobs22,
title = {Quadratic points on non-split Cartan modular curves},
author      = {Michaud-Jacobs{\relax\ (Michaud-Rodgers)}, Philippe},
journal = {International Journal of Number Theory},
volume = {18},
number = {02},
pages = {245-267},
year = {2022},
doi = {10.1142/S1793042122500178},
URL = {https://doi.org/10.1142/S1793042122500178},
eprint = {https://doi.org/10.1142/S1793042122500178}
}

@misc{lmfdb,
  shorthand    = {LMFDB},
  author       = {The {LMFDB Collaboration}},
  title        = {The {L}-functions and modular forms database},
  howpublished = {\url{https://www.lmfdb.org}},
  year         = {2026},
  note         = {[Online; accessed 25 January 2026]},
}

@article {jordan86,
    AUTHOR = {Jordan, Bruce W.},
     TITLE = {Points on {S}himura curves rational over number fields},
   JOURNAL = {J. Reine Angew. Math.},
  FJOURNAL = {Journal f\"ur die Reine und Angewandte Mathematik. [Crelle's
              Journal]},
    VOLUME = {371},
      YEAR = {1986},
     PAGES = {92--114},
      ISSN = {0075-4102,1435-5345},
   MRCLASS = {11G15 (11G18 14G25 14K22)},
  MRNUMBER = {859321},
MRREVIEWER = {Autorreferat},
       DOI = {10.1515/crll.1986.371.92},
       URL = {https://doi.org/10.1515/crll.1986.371.92},
}

@article{schembri2019examples,
  title={Examples of genuine QM abelian surfaces which are modular},
  author={Schembri, Ciaran},
  journal={Research in Number Theory},
  volume={5},
  number={1},
  pages={11},
  year={2019},
  publisher={Springer}
}

@unpublished{schembri19thesis,
           title = {Modularity of abelian surfaces over imaginary quadratic fields},
          school = {University of Sheffield},
          author = {Ciaran Schembri},
       publisher = {University of Sheffield},
            year = {2019},
             url = {https://etheses.whiterose.ac.uk/24703/}
}

@article {ohta74,
    AUTHOR = {Ohta, Masami},
     TITLE = {On {$l$}-adic representations of {G}alois groups obtained from
              certain two-dimensional abelian varieties},
   JOURNAL = {J. Fac. Sci. Univ. Tokyo Sect. IA Math.},
  FJOURNAL = {Journal of the Faculty of Science. University of Tokyo.
              Section IA. Mathematics},
    VOLUME = {21},
      YEAR = {1974},
     PAGES = {299--308},
      ISSN = {0040-8980},
   MRCLASS = {10D25 (14G25 14K22)},
  MRNUMBER = {419368},
MRREVIEWER = {F.\ Oort},
}

@article {sengunsiksek18,
    AUTHOR = {{\c S}eng\"un, Mehmet Haluk and Siksek, Samir},
     TITLE = {On the asymptotic {F}ermat's last theorem over number fields},
   JOURNAL = {Comment. Math. Helv.},
  FJOURNAL = {Commentarii Mathematici Helvetici. A Journal of the Swiss
              Mathematical Society},
    VOLUME = {93},
      YEAR = {2018},
    NUMBER = {2},
     PAGES = {359--375},
      ISSN = {0010-2571,1420-8946},
   MRCLASS = {11D41 (11F80)},
  MRNUMBER = {3811755},
MRREVIEWER = {B.\ Sury},
       DOI = {10.4171/CMH/437},
       URL = {https://doi-org.sire.ub.edu/10.4171/CMH/437},
}

@misc{githubGitHubDavidzywinaOpenImage,
	author = {David Zywina},
	title = {\emph{GitHub repository related to} Explicit open images for elliptic curves over {$\mathbb Q$}},
	howpublished = {\url{https://github.com/davidzywina/OpenImage}},
	year = {2023},
	note = {[Accessed 08-07-2026]},
}

@misc{githubFloritNonsplitCartans,
	author = {Enric Florit},
	title = {\emph{GitHub repository related to} Non-split Cartan curves and a characterization of fake elliptic curves},
	howpublished = {\url{https://github.com/3nr1c/nonsplit-Cartans}},
	year = {2026}
}

@misc{githubMichaudJacobsQuadCartan,
    title={\emph{Code to accompany the paper} {Q}uadratic points on non-split {C}artan modular curves},
    author = {Michaud-Jacobs{\relax\ (Michaud-Rodgers)}, Philippe},
    howpublished = {\url{https://github.com/michaud-jacobs/quad-cartan/tree/main}},
    year = {2020},
    note = {[Accessed 08-07-2026]}
}

@misc{zywina24,
      title={Explicit open images for elliptic curves over $\mathbb{Q}$}, 
      author={David Zywina},
      year={2024},
      eprint={2206.14959},
      archivePrefix={arXiv},
      primaryClass={math.NT},
      url={https://arxiv.org/abs/2206.14959}, 
}

@article {dlm22,
    AUTHOR = {Dose, Valerio and Lido, Guido and Mercuri, Pietro},
     TITLE = {Automorphisms of {C}artan modular curves of prime and
              composite level},
   JOURNAL = {Algebra Number Theory},
  FJOURNAL = {Algebra \& Number Theory},
    VOLUME = {16},
      YEAR = {2022},
    NUMBER = {6},
     PAGES = {1423--1461},
      ISSN = {1937-0652,1944-7833},
   MRCLASS = {11G18 (11G05 11G15 11G30 14G35)},
  MRNUMBER = {4488580},
MRREVIEWER = {David\ Ter-Borch Gram Lilienfeldt},
       DOI = {10.2140/ant.2022.16.1423},
       URL = {https://doi-org.sire.ub.edu/10.2140/ant.2022.16.1423},
}

@article {kl89,
    AUTHOR = {Kolyvagin, V. A. and Logach\"ev, D. Yu.},
     TITLE = {Finiteness of the {S}hafarevich-{T}ate group and the group of
              rational points for some modular abelian varieties},
   JOURNAL = {Algebra i Analiz},
  FJOURNAL = {Algebra i Analiz},
    VOLUME = {1},
      YEAR = {1989},
    NUMBER = {5},
     PAGES = {171--196},
      ISSN = {0234-0852},
   MRCLASS = {11G10 (11G40 14K15)},
  MRNUMBER = {1036843},
MRREVIEWER = {Takeshi\ Ooe},
}

@article {MR5057874,
    AUTHOR = {Kara, Yasemin and Nomden, Stef and \"Ozman, Ekin},
     TITLE = {Non-trivial solutions of {$Aa^p+Bb^p=Cc^3$} over number
              fields},
   JOURNAL = {J. Number Theory},
  FJOURNAL = {Journal of Number Theory},
    VOLUME = {286},
      YEAR = {2026},
     PAGES = {108--130},
      ISSN = {0022-314X,1096-1658},
   MRCLASS = {11D41 (11F75 11F80 11G05)},
  MRNUMBER = {5057874},
       DOI = {10.1016/j.jnt.2026.02.003},
       URL = {https://doi-org.sire.ub.edu/10.1016/j.jnt.2026.02.003},
}

@article {MR4620323,
    AUTHOR = {Isik, Erman and Kara, Yasemin and Ozman, Ekin},
     TITLE = {On ternary {D}iophantine equations of signature {$(p,p,3)$}
              over number fields},
   JOURNAL = {Canad. J. Math.},
  FJOURNAL = {Canadian Journal of Mathematics. Journal Canadien de
              Math\'ematiques},
    VOLUME = {75},
      YEAR = {2023},
    NUMBER = {4},
     PAGES = {1293--1313},
      ISSN = {0008-414X,1496-4279},
   MRCLASS = {11D41 (11F80)},
  MRNUMBER = {4620323},
MRREVIEWER = {Victor\ Y.\ Wang},
       DOI = {10.4153/s0008414x22000311},
       URL = {https://doi-org.sire.ub.edu/10.4153/s0008414x22000311},
}

@article {MR4101578,
    AUTHOR = {Kara, Yasemin and Ozman, Ekin},
     TITLE = {Asymptotic generalized {F}ermat's last theorem over number
              fields},
   JOURNAL = {Int. J. Number Theory},
  FJOURNAL = {International Journal of Number Theory},
    VOLUME = {16},
      YEAR = {2020},
    NUMBER = {5},
     PAGES = {907--924},
      ISSN = {1793-0421,1793-7310},
   MRCLASS = {11D41 (11F80)},
  MRNUMBER = {4101578},
MRREVIEWER = {George\ Catalin\ \c Turca\c s},
       DOI = {10.1142/S1793042120500463},
       URL = {https://doi-org.sire.ub.edu/10.1142/S1793042120500463},
}

@article {serre72,
    AUTHOR = {Serre, Jean-Pierre},
     TITLE = {Propri\'et\'es galoisiennes des points d'ordre fini des
              courbes elliptiques},
   JOURNAL = {Invent. Math.},
  FJOURNAL = {Inventiones Mathematicae},
    VOLUME = {15},
      YEAR = {1972},
    NUMBER = {4},
     PAGES = {259--331},
      ISSN = {0020-9910,1432-1297},
   MRCLASS = {14G25 (14K15)},
  MRNUMBER = {387283},
MRREVIEWER = {J.\ W. S. Cassels},
       DOI = {10.1007/BF01405086},
       URL = {https://doi-org.sire.ub.edu/10.1007/BF01405086},
}

@article {lv14,
    AUTHOR = {Larson, Eric and Vaintrob, Dmitry},
     TITLE = {Determinants of subquotients of {G}alois representations
              associated with abelian varieties},
      NOTE = {With an appendix by Brian Conrad},
   JOURNAL = {J. Inst. Math. Jussieu},
  FJOURNAL = {Journal of the Institute of Mathematics of Jussieu. JIMJ.
              Journal de l'Institut de Math\'ematiques de Jussieu},
    VOLUME = {13},
      YEAR = {2014},
    NUMBER = {3},
     PAGES = {517--559},
      ISSN = {1474-7480,1475-3030},
   MRCLASS = {11G10 (11G05 14K15)},
  MRNUMBER = {3211798},
MRREVIEWER = {Joseph\ H.\ Silverman},
       DOI = {10.1017/S1474748013000182},
       URL = {https://doi-org.sire.ub.edu/10.1017/S1474748013000182},
}

@article {hs91,
    AUTHOR = {Harris, Joe and Silverman, Joe},
     TITLE = {Bielliptic curves and symmetric products},
   JOURNAL = {Proc. Amer. Math. Soc.},
  FJOURNAL = {Proceedings of the American Mathematical Society},
    VOLUME = {112},
      YEAR = {1991},
    NUMBER = {2},
     PAGES = {347--356},
      ISSN = {0002-9939,1088-6826},
   MRCLASS = {11G30 (14H25)},
  MRNUMBER = {1055774},
MRREVIEWER = {Sheldon\ Kamienny},
       DOI = {10.2307/2048726},
       URL = {https://doi-org.sire.ub.edu/10.2307/2048726},
}

@article {dp13,
    AUTHOR = {Demeyer, Jeroen and Perucca, Antonella},
     TITLE = {The constant of the support problem for abelian varieties},
   JOURNAL = {J. Number Theory},
  FJOURNAL = {Journal of Number Theory},
    VOLUME = {133},
      YEAR = {2013},
    NUMBER = {9},
     PAGES = {2843--2856},
      ISSN = {0022-314X,1096-1658},
   MRCLASS = {14K15 (11G10 14L10 16H10 16S50)},
  MRNUMBER = {3057049},
MRREVIEWER = {Alessandra\ Bertapelle},
       DOI = {10.1016/j.jnt.2013.01.013},
       URL = {https://doi-org.sire.ub.edu/10.1016/j.jnt.2013.01.013},
}

@article {siksek09,
    AUTHOR = {Siksek, Samir},
     TITLE = {Chabauty for symmetric powers of curves},
   JOURNAL = {Algebra Number Theory},
  FJOURNAL = {Algebra \& Number Theory},
    VOLUME = {3},
      YEAR = {2009},
    NUMBER = {2},
     PAGES = {209--236},
      ISSN = {1937-0652,1944-7833},
   MRCLASS = {11G35 (11G30 14H40)},
  MRNUMBER = {2491943},
MRREVIEWER = {Oscar\ G.\ Villareal},
       DOI = {10.2140/ant.2009.3.209},
       URL = {https://doi-org.sire.ub.edu/10.2140/ant.2009.3.209},
}

@article {magma,
    AUTHOR = {Bosma, Wieb and Cannon, John and Playoust, Catherine},
     TITLE = {The {M}agma algebra system. {I}. {T}he user language},
      NOTE = {Computational algebra and number theory (London, 1993)},
   JOURNAL = {J. Symbolic Comput.},
  FJOURNAL = {Journal of Symbolic Computation},
    VOLUME = {24},
      YEAR = {1997},
    NUMBER = {3-4},
     PAGES = {235--265},
      ISSN = {0747-7171},
   MRCLASS = {68Q40},
  MRNUMBER = {MR1484478},
       DOI = {10.1006/jsco.1996.0125},
       URL = {http://dx.doi.org/10.1006/jsco.1996.0125},
}

@book {serre89,
    AUTHOR = {Serre, Jean-Pierre},
     TITLE = {Lectures on the {M}ordell-{W}eil theorem},
    SERIES = {Aspects of Mathematics},
    VOLUME = {E15},
 PUBLISHER = {Friedr. Vieweg \& Sohn, Braunschweig},
      YEAR = {1989},
     PAGES = {x+218},
      ISBN = {3-528-08968-7},
   MRCLASS = {11G10 (11D41 11G30 14Gxx)},
  MRNUMBER = {1002324},
MRREVIEWER = {Joseph\ H.\ Silverman},
       DOI = {10.1007/978-3-663-14060-3},
       URL = {https://doi-org.sire.ub.edu/10.1007/978-3-663-14060-3},
}

@article {lv14b,
    AUTHOR = {Larson, Eric and Vaintrob, Dmitry},
     TITLE = {On the surjectivity of {G}alois representations associated to
              elliptic curves over number fields},
   JOURNAL = {Bull. Lond. Math. Soc.},
  FJOURNAL = {Bulletin of the London Mathematical Society},
    VOLUME = {46},
      YEAR = {2014},
    NUMBER = {1},
     PAGES = {197--209},
      ISSN = {0024-6093,1469-2120},
   MRCLASS = {11G05},
  MRNUMBER = {3161774},
MRREVIEWER = {Laura\ Paladino},
       DOI = {10.1112/blms/bdt081},
       URL = {https://doi-org.sire.ub.edu/10.1112/blms/bdt081},
}

@article {chi87,
    AUTHOR = {Chi, W\^en Ch\^en},
     TITLE = {Twists of central simple algebras and endomorphism algebras of
              some abelian varieties},
   JOURNAL = {Math. Ann.},
  FJOURNAL = {Mathematische Annalen},
    VOLUME = {276},
      YEAR = {1987},
    NUMBER = {4},
     PAGES = {615--632},
      ISSN = {0025-5831,1432-1807},
   MRCLASS = {14K15 (11G10)},
  MRNUMBER = {879540},
MRREVIEWER = {Ching-li\ Chai},
       DOI = {10.1007/BF01456990},
       URL = {https://doi.org/10.1007/BF01456990},
}

@incollection {pyle04,
    AUTHOR = {Pyle, Elisabeth E.},
     TITLE = {Abelian varieties over {$\mathbb Q$} with large endomorphism
              algebras and their simple components over {$\overline{\mathbb Q}$}},
 BOOKTITLE = {Modular curves and abelian varieties},
    SERIES = {Progr. Math.},
    VOLUME = {224},
     PAGES = {189--239},
 PUBLISHER = {Birkh\"auser, Basel},
      YEAR = {2004},
      ISBN = {3-7643-6586-2},
   MRCLASS = {11G10 (14K15)},
  MRNUMBER = {2058652},
MRREVIEWER = {Sigrid\ Wortmann},
}

@article {quer09,
    AUTHOR = {Quer, Jordi},
     TITLE = {Fields of definition of building blocks},
   JOURNAL = {Math. Comp.},
  FJOURNAL = {Mathematics of Computation},
    VOLUME = {78},
      YEAR = {2009},
    NUMBER = {265},
     PAGES = {537--554},
      ISSN = {0025-5718,1088-6842},
   MRCLASS = {11G10 (11G18 11R34)},
  MRNUMBER = {2448720},
MRREVIEWER = {Ahmad\ El-Guindy},
       DOI = {10.1090/S0025-5718-08-02132-7},
       URL = {https://doi-org.sire.ub.edu/10.1090/S0025-5718-08-02132-7},
}

@article {gq14,
    AUTHOR = {Guitart, Xavier and Quer, Jordi},
     TITLE = {Modular abelian varieties over number fields},
   JOURNAL = {Canad. J. Math.},
  FJOURNAL = {Canadian Journal of Mathematics. Journal Canadien de
              Math\'ematiques},
    VOLUME = {66},
      YEAR = {2014},
    NUMBER = {1},
     PAGES = {170--196},
      ISSN = {0008-414X,1496-4279},
   MRCLASS = {11G10 (11F11 11G18)},
  MRNUMBER = {3150707},
MRREVIEWER = {Ahmad\ El-Guindy},
       DOI = {10.4153/CJM-2012-040-2},
       URL = {https://doi-org.sire.ub.edu/10.4153/CJM-2012-040-2},
}

@article {bdmtv23,
    AUTHOR = {Balakrishnan, Jennifer S. and Dogra, Netan and M\"uller, J.
              Steffen and Tuitman, Jan and Vonk, Jan},
     TITLE = {Quadratic {C}habauty for modular curves: algorithms and
              examples},
   JOURNAL = {Compos. Math.},
  FJOURNAL = {Compositio Mathematica},
    VOLUME = {159},
      YEAR = {2023},
    NUMBER = {6},
     PAGES = {1111--1152},
      ISSN = {0010-437X,1570-5846},
   MRCLASS = {11G18 (11G50 11Y50 14G05 14G35)},
  MRNUMBER = {4589060},
MRREVIEWER = {Steven\ D.\ Galbraith},
       DOI = {10.1112/s0010437x23007170},
       URL = {https://doi-org.sire.ub.edu/10.1112/s0010437x23007170},
}

@article {frengley23,
    AUTHOR = {Frengley, Sam},
     TITLE = {Congruences of elliptic curves arising from nonsurjective
              {${\rm mod} \,N$} {G}alois representations},
   JOURNAL = {Math. Comp.},
  FJOURNAL = {Mathematics of Computation},
    VOLUME = {92},
      YEAR = {2023},
    NUMBER = {339},
     PAGES = {409--450},
      ISSN = {0025-5718,1088-6842},
   MRCLASS = {11G05 (11F80)},
  MRNUMBER = {4496970},
MRREVIEWER = {Riccardo\ Pengo},
       DOI = {10.1090/mcom/3770},
       URL = {https://doi-org.sire.ub.edu/10.1090/mcom/3770},
}

@misc{cn23,
      title={On the modularity of elliptic curves over imaginary quadratic fields}, 
      author={Ana Caraiani and James Newton},
      year={2025},
      eprint={2301.10509},
      archivePrefix={arXiv},
      primaryClass={math.NT},
      url={https://arxiv.org/abs/2301.10509}, 
}

\appendix
\section{Equations and points on \texorpdfstring{$X_{ns}(15)$}{Xns(15)}}
\label{appendix:Xns15}

The following set of equations defines the curve $X_{ns}(15)$ in $\PP^6$:
\begin{align*}
& -3x_1x_3 + 3x_1x_4 - 6x_1x_6 + 6x_1x_7 + 2x_2x_3 + 2x_2x_5 = 0, \\[1ex]
& 4x_1x_3 - 6x_1x_4 + 4x_1x_5 - 3x_1x_6 - 6x_1x_7 - x_2x_3 + x_2x_7 = 0, \\[1ex]
& 3x_1x_3 - 6x_1x_4 - 6x_1x_5 - 3x_1x_7 - 2x_2x_3 + x_2x_4 - 4x_2x_5 + x_2x_6 + x_2x_7 = 0, \\[1ex]
& 7x_1^2 + 2x_1x_2 - x_3^2 + x_3x_4 - 3x_3x_5 + x_3x_6 + x_4^2 + x_4x_5 - 3x_5^2 \\ 
&\qquad + 2x_5x_7 + x_6^2 - x_6x_7 + x_7^2 = 0, \\[1ex]
& x_1^2 - 4x_1x_2 - x_3^2 + x_3x_6 + x_4^2 - 3x_4x_6 + 2x_4x_7 + 3x_5^2 \\ 
&\qquad + 6x_5x_6 - 2x_5x_7 - 2x_6^2 - 2x_6x_7 = 0, \\[1ex]
& -2x_1^2 - 2x_1x_2 - x_3^2 - 3x_3x_5 + x_3x_7 + 3x_4x_5 - 3x_4x_6 - 2x_5^2 \\ 
&\qquad + 4x_5x_6 + 2x_5x_7 - x_6^2 - 2x_6x_7 - x_7^2 = 0, \\[1ex]
& 3x_1^2 - 2x_1x_2 - x_3^2 - 3x_3x_4 - x_3x_5 - 4x_3x_6 + 2x_4^2 - x_4x_5 - 4x_4x_6 \\ 
&\qquad - 3x_5^2 + 2x_5x_6 - 2x_5x_7 - 2x_6x_7 - 2x_7^2 = 0, \\[1ex]
& -3x_1^2 + 2x_1x_2 + x_3x_4 + 3x_3x_5 + 2x_3x_6 - 2x_3x_7 + 3x_4x_5 - x_4x_6 \\ 
&\qquad + x_4x_7 - x_5^2 + 6x_5x_6 - 2x_5x_7 - x_6^2 + x_7^2 = 0, \\[1ex]
& x_1^2 + 4x_2^2 + x_3^2 - 2x_3x_4 - 4x_3x_5 - 6x_3x_6 - 5x_4^2 + 2x_4x_5 \\ 
&\qquad - 2x_4x_6 - 4x_4x_7 + 7x_5^2 + 4x_5x_6 - 4x_5x_7 + 8x_6^2 - 4x_6x_7 = 0,\\
& 2x_1^2 + 2x_1x_2 + 2x_3^2 - x_3x_4 + 3x_3x_5 - 2x_3x_6 - 2x_4^2 + x_4x_5 \\ 
&\qquad - x_4x_6 - 3x_4x_7 + 2x_5^2 + 2x_5x_6 + 2x_5x_7 - x_6^2 + 2x_6x_7 - 3x_7^2 = 0.
\end{align*}
\newpage
\setcounter{table}{0}
\renewcommand{\thetable}{A.\arabic{table}}
\renewcommand*{\theHtable}{\thetable}
Table~\ref{table:Xns15-known} shows the 14 pairs of known quadratic points on $X_{ns}(15)$, this set is shown to be complete in Theorem~\ref{theorem:Xns15}. The morphism $\varrho_w:X_{ns}(15)\to X_{ns}^+(15)$ was introduced in Section~\ref{section:non-split Cartan curves}.

\begin{table}[h!]
\makebox[\textwidth]{
\centering
\renewcommand{\arraystretch}{1.5}
\begin{tabular}{|c|c|c|c|c|}
\hline
$Q_i$ & $\theta^2$ & Coordinates & CM & $P_i=\varrho_w(Q_i)$ \\ \hline
$Q_{1}$ & $-3$ & $\left(-\frac{3}{16}\theta : -\frac{3}{4}\theta : -\frac{1}{8} : -\frac{5}{8} : \frac{5}{16} : 1 : 1\right)$ & $-3$ & $(1 : -2 : 1)$ \\ \hline
$Q_{2}$ & $-3$ & $\left(\frac{3}{4}\theta : -\frac{3}{4}\theta : -2 : -\frac{5}{2} : \frac{5}{4} : 1 : 1\right)$ & $-3$ & $(-2 : -2 : 1)$ \\ \hline
$Q_{3}$ & $-7$ & $\left(\frac{1}{6}\theta : -\theta : \frac{4}{3} : 0 : -\frac{5}{6} : -\frac{2}{3} : 1\right)$ & $-7$ & $(-1 : -1 : 1)$ \\ \hline
$Q_{4}$ & $-7$ & $\left(-\frac{1}{6}\theta : -\frac{3}{2}\theta : -\frac{1}{3} : 0 : \frac{5}{6} : \frac{8}{3} : 1\right)$ & $-7$ & $(0 : -1 : 1)$ \\ \hline
$Q_{5}$ & $-7$ & $\left(-\frac{3}{22}\theta : -\frac{6}{11}\theta : \frac{8}{11} : -\frac{10}{11} : \frac{5}{22} : \frac{16}{11} : 1\right)$ & $-7$ & $(1 : 1 : 1)$ \\ \hline
$Q_{6}$ & $-7$ & $\left(\frac{3}{2}\theta : -\frac{3}{2}\theta : 13 : 10 : -\frac{5}{2} : -4 : 1\right)$ & $-7$ & $(-2 : 1 : 1)$ \\ \hline
$Q_{7}$ & $-43$ & $\left(-\frac{1}{12}\theta : -\frac{1}{8}\theta : -\frac{7}{6} : -\frac{5}{4} : \frac{5}{12} : \frac{7}{12} : 1\right)$ & $-43$ & $(1 : -1 : 0)$ \\ \hline
$Q_{8}$ & $-43$ & $\left(-\frac{1}{18}\theta : -\frac{1}{12}\theta : \frac{2}{9} : -\frac{5}{6} : \frac{5}{18} : \frac{13}{18} : 1\right)$ & $-43$ & $(1 : 1 : 0)$ \\ \hline
$Q_{9}$ & $-67$ & $\left(\frac{1}{24}\theta : -\frac{1}{4}\theta : \frac{11}{12} : -\frac{5}{4} : -\frac{25}{24} : \frac{1}{6} : 1\right)$ & $-67$ & $(-1 : 0 : 1)$ \\ \hline
$Q_{10}$ & $-67$ & $\left(-\frac{1}{24}\theta : -\frac{3}{8}\theta : -\frac{7}{6} : \frac{5}{4} : \frac{25}{24} : \frac{11}{6} : 1\right)$ & $-67$ & $(0 : 0 : 1)$ \\ \hline
$Q_{11}$ & $-163$ & $\left(-\frac{7}{264}\theta : -\frac{3}{44}\theta : \frac{41}{132} : -\frac{35}{44} : \frac{65}{264} : \frac{61}{66} : 1\right)$ & $-163$ & $(3 : 36 : 1)$ \\ \hline
$Q_{12}$ & $-163$ & $\left(-\frac{7}{96}\theta : -\frac{1}{32}\theta : -\frac{53}{24} : -\frac{35}{16} : \frac{65}{96} : \frac{19}{24} : 1\right)$ & $-163$ & $(-4 : 36 : 1)$ \\ \hline
$Q_{13}$ & $-214663$ & $\left(-\frac{7}{8286}\theta : -\frac{3}{1381}\theta : -\frac{2836}{4143} : -\frac{1140}{1381} : \frac{2795}{8286} : \frac{2618}{4143} : 1\right)$ & no & $(3 : -37 : 1)$ \\ \hline
$Q_{14}$ & $-214663$ & $\left(-\frac{7}{7194}\theta : -\frac{1}{2398}\theta : \frac{41}{3597} : -\frac{1140}{1199} : \frac{2795}{7194} : \frac{2072}{3597} : 1\right)$ & no & $(-4 : -37 : 1)$ \\ \hline
\end{tabular}
}
\vspace{0.5em}
\caption{Known points on $X_{ns}(15)$, coming from pullbacks of the points on $X_{ns}^+(15)(\QQ)$.}
\label{table:Xns15-known}
\end{table}

\end{document}